\documentclass[11pt,reqno]{amsart}
\usepackage{amsmath,amssymb,mathrsfs,extarrows,mathtools,amsthm}
\allowdisplaybreaks[3]
\usepackage{graphicx,cite}
\usepackage{subfigure}
\usepackage{xcolor}
\usepackage{epstopdf}
\usepackage{graphics} %% add this and next lines if pictures should be in esp format
\usepackage{epsfig} %For pictures: screened artwork should be set up with an 85 or 100 line screen
\usepackage{tikz}
\usepackage{paralist}
\usepackage{booktabs}
\usepackage{enumerate}
\usepackage{enumitem}
\usepackage{bbm}
\usepackage{mdwlist}
\usepackage{url}
\usepackage{algorithm,algorithmic}
\makeatletter

\newcommand{\Rmnum}[1]{\expandafter\@slowromancap\romannumeral #1@}
\makeatother
\newtheorem{theorem}{Theorem}[section]

\newtheorem{lemma}[theorem]{Lemma}
\newtheorem{proposition}[theorem]{Proposition}
\newtheorem{remark}[theorem]{Remark}

\newcommand{\face}{\Gamma_0}
\newcommand{\lowsp}{\mathbb{R}^2_-}
\newcommand{\upsp}{\mathbb{R}^2_+}

\newcommand{\my}{\mathcal{Y}}
\newcommand{\mg}{\mathcal{G}}

\newcommand{\me}{\mathrm{e}}
\newcommand{\mi}{\mathrm{i}}

\newcommand{\mv}{\mathrm{v}}
\newcommand{\ms}{\mathbb{S}}
\newcommand{\ps}{\widehat{\mathbb{S}}}
\newcommand{\mm}{\mathbb{M}}
\newcommand{\mr}{\mathbb{R}^2}

\newcommand{\oo}{\mathcal{O}}

\newcommand{\uinc}{u^i}
\newcommand{\uincf}{\uinc_f}
\newcommand{\usc}{u^{\rm sc}}
\begin{document}
%	\maketitle	
	\title[Selective focusing in a layered medium]{Selective focusing of multiple particles in a layered medium}

	\author{Jun Lai}
	\address{School of Mathematical Sciences, Zhejiang University,
		Hangzhou, Zhejiang 310027, China}
	\email{laijun6@zju.edu.cn}
	\author{Jinrui Zhang}
	\address{School of Mathematical Sciences, Zhejiang University,
		Hangzhou, Zhejiang 310027, China}
	\email{zhangjinrui@zju.edu.cn}

	
	% REQUIRED
	\begin{abstract}
		Inverse scattering in layered media has a wide range of applications, examples including geophysical exploration, medical imaging, and remote sensing.  In this paper, we develop a selective focusing method for identifying multiple unknown buried scatterers in a layered medium. The method is derived through the asymptotic analysis of the time reversal operator using the layered Green's function and limited aperture measurements. We begin by showing the global focusing property of the time reversal operator. Then we demonstrate that each small sound-soft particle gives rise to one  significant eigenvalue of the time reversal operator, while each sound-hard particle gives three. The associated eigenfunction generates an incident wave focusing selectively on the corresponding unknown particle. Finally, we employ the time reversal method as an initial indicator and propose an effective Bayesian inversion scheme to reconstruct multiple buried extended scatterers for enhanced resolution. Numerical experiments are provided to demonstrate the efficiency.
	\end{abstract}
	\maketitle
	% REQUIRED
%	\begin{keywords}
%		Inverse scattering, layered medium scattering, time reversal operator, selective focusing, Bayesian inversion
%	\end{keywords}
%	
%	% REQUIRED
%	\begin{MSCcodes}
%		78A46, 35B40, 35R30, 31A10
%	\end{MSCcodes}
	
	\section{Introduction}
	Inverse scattering in layered media has attracted  lots of attention due to its various important applications, including non-destructive testing, geophysical exploration, medical imaging, and many others \cite{delbaryInverseElectromagneticScattering2007,baoTimeHarmonicAcousticScattering2018,wangtaoluWellPosednessUPMLMethod2024}. However, there are several difficulties arising from the inverse scattering problems in layered media.  One is that the measurements are typically taken over a limited aperture, which restricts the amount of information available for reconstructing the scatterers. Another is that the layered Green's function, which is essential for analyzing the inverse problem, is given as a Sommerfeld integral and lacks a closed form. In this paper, we focus on the inverse scattering in a two-layered medium and propose an effective selective focusing method for imaging multiple buried obstacles.
	
	A variety of inversion algorithms have been proposed to solve inverse scattering problems in a layered medium. These methods can be broadly  divided into iterative-type and direct imaging methods. Iterative methods make use of the nonlinear constrained optimization techniques with suitably chosen regularization terms to tackle the ill-posedness  \cite{baoInverseScatteringProblems2015,borgesInverseScatteringReconstruction2020}. While these methods are effective in achieving high resolution reconstructions of unknown scatterers, they are computationally intensive and often require prior knowledge of the scatterer's physical and geometric properties. On the other hand, direct imaging methods have been developed for layered media to avoid expensive iterations~\cite{parkImagingThinDielectric2010,gebauerSamplingMethodDetecting2005,yangDetectingBuriedWavepenetrable2018}. For instance, in \cite{ammariMUSICAlgorithmLocating2005}, the MUSIC (MUltiple SIgnal Classification) algorithm was studied to determine the number and locations of small inclusions buried in the lower half space.  In \cite{liRecoveringMultiscaleBuried2015}, a direct imaging method was studied to recover multiscale buried anomalies with a prior assumption on the shape of extended obstacles. Despite their computational efficiency, direct imaging methods typically provide only low resolution reconstructions in practice, especially in layered media~\cite{liImagingBuriedObstacles2021}.  
	
	In this paper, we develop a hybrid imaging method for the inverse obstacle scattering problems in layered media based on the DORT (French acronym for Decomposition of the Time Reversal Operator) model. DORT method was first developed in \cite{pradaEigenmodesTimeReversal1994} in the context of ultrasonics as a time reversal technique \cite{liDecompositionTimeReversal2012a}. Due to its robustness to noise, this method has been employed in a wide range of applications \cite{pradaTimeReversalTechniques2002}. In the process of the DORT method, one first uses a time reversal mirror, composed of an array of transducers at infinity, to emit an incident wave into the medium containing some unknown obstacles, then conjugates the measured far field and reemits. The time reversal operator $T$ is obtained by iterating this procedure twice. It turns out that the eigensystem of the time reversal operator carries important information on the scatterers contained in the propagation medium.  A rigorous justification of the DORT method was given in \cite{hazardSelectiveAcousticFocusing2004} for the acoustic scattering problem by small sound-soft scatterers. It has been extended to electromagnetics with perfectly conducting boundary conditions ~\cite{antoineFarFieldModeling2008}, small dielectric inhomogeneities \cite{burkardFarFieldModel2013}, and the inverse elastic scattering problems~\cite{laiFastInverseElastic2022,zhangSelectiveFocusingElastic2023}.
	However, all current mathematical justifications are presented in the case of homogeneous media. Due to the difficulties posed by limited aperture and complicated asymptotic analysis,  the problem for layered media is still open. Unlike the homogeneous case, the far field operator with limited aperture in layered media is no longer normal \cite{hazardSelectiveAcousticFocusing2004}. It makes the spectral structure of the time reversal operator differ from that of the far field operator and consequently the spectral analysis becomes much more involved. To bridge this gap, we are concerned with the time-harmonic inverse scattering by multiple buried particles in a layered medium and give a detailed spectral analysis of the time reversal operator through the asymptotic behavior of the layered Green’s function.
	
	In particular, using a limited aperture time-harmonic far field model, we prove that for $a\ll k_-^{-1}\ll L$, where $a$ denotes the size of the scatterer, $k_-$ the wavenumber of the lower half-space and $L$ the minimal distance between the scatterers, each buried sound-soft particle gives rise to one significant eigenvalue while each sound-hard particle gives rise to three eigenvalues. Meanwhile, each associated eigenfunction generates an incident wave that selectively focuses on the corresponding unknown buried scatterers. Furthermore, with the time reversal method as an initial indicator, we propose an effective Bayesian inversion technique \cite{bui-thanhAnalysisInfiniteDimensional2014} to recover multiple extended obstacles for higher resolution in a layered  medium. It is worth mentioning that the asymptotic behavior in the present work is related to MUSIC-type imaging methods developed by Ammari et al~\cite{ammariMUSICAlgorithmLocating2005}. However, instead of projecting onto the noise space, here we explicitly construct Herglotz waves that can focus globally on extended obstacles and selectively on small particles based on the detailed spectral analysis of the time reversal operator. It is physically easier to implement and can be used to achieve energy focusing, while MUSIC-type methods are mainly algorithmic for imaging small particles only.
    
	The paper is organized as follows. In Section 2, we formulate the inverse scattering problems of multiple buried particles in a layered medium in two dimensions. In Section 3, we show the property of global focusing by using the eigenfunctions of the time reversal operator for extended obstacles. Section 4
	is devoted to the selective focusing property of the layered time reversal operator in the imaging of small and distant particles. In Section 5, an effective Bayesian approach is developed to reconstruct multiple buried extended obstacles with the time reversal method as an initial guess. 
	Numerical experiments are presented in Section 6 to illustrate the efficiency of our hybrid method. Finally, the paper is concluded in Section 7.
	
	\section{Problem formulation}
	For $x=(x_1,x_2)\in\mr$, we denote $\lowsp=\{(x_1,x_2)\in\mr|x_2<0\}$ and $\upsp=\{(x_1,x_2)\in\mr|x_2>0\}$ the lower and upper half-spaces, respectively. The interface between the two layers is denoted by $\face=\{(x_1,x_2)\in\mr|x_2=0\}$. 
	
	Consider $M$ well separated impenetrable particles fully embedded in the lower half-space $\lowsp$, denoted by $D_1,D_2,\dots, D_M$. We assume that their boundaries $\Gamma_{1}, \Gamma_{2}, \dots, \Gamma_{M}$ are at least $C^2$ smooth. Let $D=D_{1}\cup D_{2}\cup\cdots\cup D_{M}$ and $\Gamma=\Gamma_{1}\cup\Gamma_{2}\cup\cdots\cup\Gamma_{M}$.  Denote $\nu$ the unit exterior normal vector on $\Gamma$. It also denotes the unit upward normal vector on $\face$, if no ambiguity arises.
	\begin{figure}[!htp]
		\centering
		\includegraphics[scale=0.4]{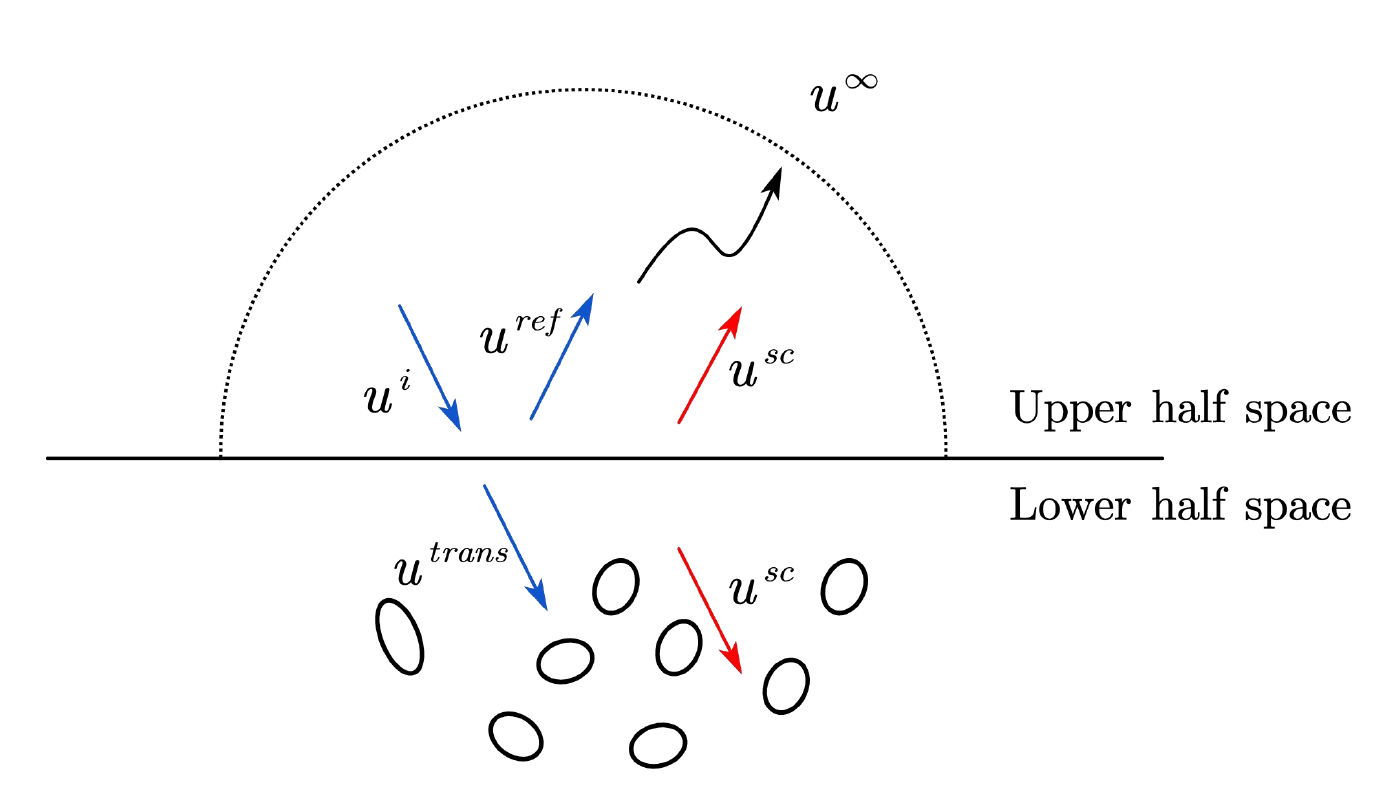}
		\caption{Inverse scattering of multiple buried particles using the far field measurement.}\label{ex11}
	\end{figure} 
	Let the particles be illuminated by a time-harmonic plane wave $\uinc=\me^{\mi k_+x\cdot d}$ propagating in the direction 
	\[{d}= (d_1,d_2)\in \ps,\]
	% \[d = (\cos \theta_d,\sin\theta_d),\qquad \theta_d\in(\pi+\theta_c,2\pi-\theta_c),\]
	where \begin{align*}
		\ps=\{(\cos \theta,\sin\theta)|\theta\in(\pi+\theta_c,2\pi-\theta_c)\},
 \end{align*}
	and $\theta_c$ is defined as
	\[\theta_{c}:=\begin{cases}\arccos(k_{-}/k_{+}),&k_{+}>k_{-},\\ 0,&k_{+}<k_{-},\end{cases}\]
	with $k_{\pm}>0$ being the wavenumbers in $\mr_\pm$, respectively. An illustration is given in Figure \ref{ex11}.
	
	The total field $u$ consists of the background field $u^{bk}$ and the scattered field $\usc$, i.e., $ u=u^{bk}+\usc$. From the Fresnel formula, $u^{bk}$ is given by 
	\[u^{bk}(x,d,k_+,k_-)=\begin{cases}\me^{\mi k_+x\cdot d}+u^{ref}(x,d,k_+,k_-),&x\in\mathbb{R}_+^2,\\ u^{trans}(x,d,k_+,k_-),&x\in\mathbb{R}_-^2,\end{cases}\]
	with 
	\[u^{ref}(x,d,k_+,k_-):=R(d)\me^{\mi k+x\cdot r(d)},\quad u^{trans}(x,d,k_+,k_-):=Q(d)\me^{\mi k_-x\cdot \mv(d)},\]
	where 
	\begin{align}\label{mumu}
		r(d)=(d_1,-d_2),\qquad \mv(d)=\left(\mu d_1,\mathrm{sign}(d_2)\sqrt{1-\mu^2d_1^2}\right) \mbox{ with } \mu = {k_+}/{k_-},
	\end{align}
	are the reflection and transmission directions, respectively. 
	%	 with $\theta_d^t\in(\pi,2\pi)$ satisfying that
	%	$k_{+}\cos\theta_{d} =k_-\cos\theta_{d}^{t}$,
	The corresponding reflection and transmission coefficients $R(d)$ and $Q(d)$ are given by
	\begin{align}\label{transcoef}
		R(d)=\frac{\mu d_2-\mv_2(d)}{\mu d_2+\mv_2(d)},\quad Q(d)=\frac{2\mu d_2}{\mu d_2+\mv_2(d)}.
	\end{align}

	With preparations above, the forward problem of the scattering due to the buried impenetrable obstacles in a layered medium can be formulated as finding the scattered wave field 
	$u^s\in H^1_{loc}(\mathbb{R}^2\setminus\overline{D})$ such that 
	\begin{equation}\label{scatteredfield}
		\begin{cases}
			\Delta\usc+(k_{\pm})^2\usc=0\quad & {\rm in}~
			\mr_{\pm}\setminus\overline{D},\\
			\left[\usc\right]=0,\left[\frac{\partial\usc}{\partial\nu}\right]=0 &{\rm on}~\face,\\
			\mathcal{B}(\usc)=-\mathcal{B}(u^{bk})&{\rm on}~\Gamma,\\
			\lim\limits_{r\to\infty}\int_{\mathbb{S}_{r,\pm}}\left|\frac{\partial \usc}{\partial r}-\mi k_\pm \usc\right|^2ds=0,
		\end{cases}
	\end{equation}
	where $\left[\cdot\right]$ denotes the jump across the interface $\face$, $\mathcal{B}$ denotes one of the following boundary conditions
	\begin{align*}
		\mbox{ \textit{sound-soft}: }\mathcal{B}(\usc)&=\usc\quad{\rm on}~\Gamma,\\
		\mbox{ \textit{sound-hard}: }\mathcal{B}(\usc)&=\frac{\partial \usc}{\partial \nu}\quad{\rm on}~\Gamma,
	\end{align*}
	%\[ \quad  ~,\]
	and $\mathbb{S}_{r,\pm}=\{x\in\mr_{\pm}:|x|=r\}$ denotes the half circle
	with radius $r$ centered at the origin in $\mr_{\pm}$. The well-posedness of the boundary value problem (\ref{scatteredfield}) follows a similar spirit in  \cite{cutzachExistenceUniquenessAnalyticity1998a}, where the case of electromagnetic scattering was considered.
	
	%In addition, the scattered field $\usc$ is required to satisfy
	%the Sommerfeld radiation condition
	%\[
	%	\lim\limits_{r\to\infty}\int_{\mathbb{S}_{r,\pm}}\left|\frac{\partial \usc}{\partial r}-\mi k_\pm \usc\right|^2ds=0,
	%\]
	%where $\mathbb{S}_{r,\pm}=\{x\in\mr_{\pm}:|x|=r\}$ is the half circle
	%of radius $r$ centered at the origin in $\mr_{\pm}$.
	
	Let $\Phi(x, z)$ be the Green’s function of \eqref{scatteredfield}.
	%the unperturbed two-half space $x=(x_1,x_2)$ excited by a point source $z=(z_1,z_2)$ in the lower half-space. 
	One can derive that it is given by \cite{ammariMUSICAlgorithmLocating2005}
	\begin{align}\label{green}
		\begin{aligned}
			\Phi(x,z)=
			\begin{cases}\Phi^{trans}(x,z),&x\in\mathbb{R}_+^2;\\ \Phi^i(x,z)+\Phi^{ref}(x,z),&x\in\mathbb{R}_-^2,x\neq z\end{cases}
		\end{aligned}
	\end{align}
	where
	\begin{align*}
		\Phi^{i}(x, z)&=
		\frac{i}{4} H_{0}^{(1)}\left(k_{-}|x-z|\right),\\
		\Phi^{trans}(x, z)&=\frac{i}{2 \pi} \int_{\mathbb{R}} \frac{e^{i\left(\eta_{+} x_2-\eta_{-} z_2\right)}}{\eta_{+}+\eta_{-}} e^{i \xi \cdot\left(x_1-z_1\right)} d \xi,\\
		\Phi^{ref}(x, z)&=\frac{i}{4 \pi} \int_{\mathbb{R}} \frac{\eta_{-}-\eta_{+}}{\eta_{-}\left(\eta_{+}+\eta_{-}\right)} e^{-i \eta_{-}\left(x_2+z_2\right)} e^{i \xi \cdot\left(x_1-z_1\right)} d \xi,
	\end{align*}
	with  
	$\eta_\pm=\sqrt{k_\pm^2-|\xi|^2}$, and $\Phi^{trans}$ and $\Phi^{ref}$ are given as Sommerfeld integrals~\cite{JGM15}.
	
	 For $x\in\upsp$, as $|x|\rightarrow\infty$,  the asymptotic behavior of the scattered field $\usc$ is given by 
	\begin{equation*}
		\usc(x,d) = \frac{\me^{\mi\pi/4}}{\sqrt{8k_+\pi}}\frac{\me^{\mi k_+ |x|}}{\sqrt{|x|}} u^{\infty}(\hat{x},d)+\mathcal{O}\left(|x|^{-\frac{3}{2}}\right),
	\end{equation*}
	where   $u^{\infty}(\hat{x},d)$ is defined on the upper unit half circle $\ms^+$ with $\hat{x}=x/|x|$ and known as the far field pattern \cite{ammariMUSICAlgorithmLocating2005}. Based on Green's representation formula and the asymptotic behavior of layered Green's function \eqref{green}, it has the integral form given by the following proposition \cite{liRecoveringMultiscaleBuried2015}.
	
	\begin{proposition}\label{farree}
		For layered medium scattering, the  far field in the upper half space is
		
		\begin{align}\label{farre}
			u^{\infty}(\hat{x},d)= Q(\hat{x})\int_{\Gamma}\Bigg\{\usc(y,d)\frac{\partial \me^{-\mi k_-\mv(\hat{x})\cdot y}}{\partial\nu(y)}-\me^{-\mi k_-\mv(\hat{x})\cdot y}\frac{\partial \usc(y,d)}{\partial\nu(y)}\Bigg\}ds(y),\quad \hat{x}\in-\ps,
		\end{align}
        where $Q(\hat{x})$ is given in equation \eqref{transcoef}.
		Furthermore, $u^{\infty}(\hat{x},d)$ is an analytic function on $-\ps\times\ps$ and satisfies the reciprocity relation
		\begin{align}\label{recp}
			u^{\infty}(\hat{x},d)=u^{\infty}(-d,-\hat{x}).
		\end{align}
	\end{proposition}
	
	The inverse problem of acoustic scattering by multiple buried impenetrable  particles is:
	\begin{itemize}
		\item Based on the far field pattern $u^{\infty}(\hat{x},d)$ from different incident directions in $\ps$,  recover the geometric information of $D_1, D_2, \cdots, D_M$, including locations and shapes.
	\end{itemize}
	
	In the following section, we will introduce the time reversal operator in a layered medium and show the global focusing property in the general case. 
	\section{Global focusing in a layered medium}
	\subsection{Time reversal operator}
	Denote $u^{\infty}(\hat{x},{d},f({d}))$ the far field pattern radiated by the incident wave  $\uinc(x,{d},f({d})) = f({d})\me^{\mi k_+ x\cdot{d}},~{d}\in \ps.$
	The Herglotz wave with kernel $f\in L^2(\ps)$ is defined by
	\begin{eqnarray}\label{herglotz}
		\uincf(x) =\int_{\ps} \me^{\mi k_+ {d}\cdot x} f({d})ds_{{d}},
	\end{eqnarray}
	which is a superposition of plane waves and satisfies the Helmholtz equation in $\upsp$. By linearity, the corresponding far field operator $${F}:L^2(\ps)  \rightarrow L^2(-\ps) $$  due to the incident wave $\uincf(x)$ is given by
	\begin{eqnarray}\label{faroperater}
		({F}f)(\hat{x}) = \int_{\ps}u^{\infty}(\hat{x},{d})f({d})ds_{d},\quad
		\hat{x}\in -\ps.
	\end{eqnarray}
    According to the analyticity of the kernel $u^{\infty}(\hat{x},{d})$ and the reciprocity relation \eqref{recp}, it holds the following result.
	\begin{lemma}
		The far field operator ${F}:L^2(\ps)  \rightarrow L^2(-\ps)$ defined by \eqref{faroperater} is a compact operator. Its adjoint operator ${F}^*:L^2(-\ps)  \rightarrow L^2(\ps)$ is given by
		\begin{align}\label{fadj}
			{F}^*f=\overline{R{F}\overline{R f}}\quad\forall f\in L^2(-\ps),
		\end{align}
		where $R$ is the symmetry operator defined by
		${R}f({d})=f(-{d}),{d}\in\ps$.
	\end{lemma}
%	\begin{proof}
%		The compactness of the integral operator $F$ follows immediately from the analyticity of its kernel $u^{\infty}(\hat{x},{d})$. The adjoint ${F}^*$ of $F$ follows the reciprocity relation \eqref{recp}.
  %       which implies
		% \[K(\hat{x},d)=\overline{u^\infty(d,\hat{x})}=\overline{u^\infty(-\hat{x},-d)}.\]
%	\end{proof}
	It is worth mentioning that the above lemma is an extension of the classical result for acoustic and electromagnetic scattering in homogeneous media \cite{coltonInverseAcousticElectromagnetic2019}.
    
	We are now able to define the time reversal operator $T$.  First, let us measure the far field of the scattered field around the aperture $-\ps$ due to the Herglotz wave $\uinc_f$ with $f\in L^2(\ps) $, and then use the conjugate of the far field pattern as the kernel $g$ of a new Herglotz wave, which implies
	\begin{eqnarray*}
		g = \overline{RFf}\in L^2(\ps).
	\end{eqnarray*}
	The symmetry operator $R$ used here is to reemit the wave from the opposite of the measured direction. Iterating this cycle twice, the time reversal operator $T$ is given by 
	\begin{eqnarray*}
		{T} f = \overline{RF g} = \overline{RF \overline{RFf}}.
	\end{eqnarray*}
	Then one can show that the time reversal operator has the form ${T}={F}^*{F}\in\mathcal{L}(L^2(\ps))$ by equation \eqref{fadj}, and following a similar argument in \cite{hazardSelectiveAcousticFocusing2004}, ${T}$ is the integral operator with kernel
	\begin{align}\label{kne}
		{t}(\beta,\alpha )=\int_{-\ps}A(\gamma,\alpha)\overline{A(\gamma,\beta)}ds_\gamma,
	\end{align}
	where $A(\cdot,\cdot)=u^\infty(\cdot,\cdot)$.
	Moreover, ${T}$ defines a compact positive and self-adjoint operator. Nevertheless, unlike the full aperture far field operator $F$ in homogeneous media \cite{hazardSelectiveAcousticFocusing2004}, the limited aperture far field operator ${F}$ is not normal anymore. Consequently, the eigenfunctions of ${F}$ and ${T}$ are not necessarily the same.
	Thus, the spectral analysis needs to be carried directly on the time reversal operator $T$ instead of the far field operator $F$. Surprisingly, as we are going to see, based on a limited aperture model and a more complicated asymptotic analysis, all focusing results in the homogeneous case \cite{hazardSelectiveAcousticFocusing2004} still hold in the layered medium scattering.
	
	Before we introduce the global focusing property, we need the following lemma.
	\begin{lemma}\label{decay}
		Define \begin{align}\label{b0}
			B_{0}(y):=\int_{\ps}{Q^{2}}(\alpha)\me^{\mi k_-\mv(\alpha)\cdot y}d s_\alpha, \quad B_{1}(y):=\int_{\ps}\mv(\alpha){Q^{2}}(\alpha)\me^{\mi k_-\mv(\alpha)\cdot y}d s_\alpha,
		\end{align}
		where $\mv(\alpha)$ and $Q(\alpha)$ are defined in \eqref{mumu} and \eqref{transcoef}, respectively. For $y\in\mr$ with $|y|$ large enough, we have
		\[|B_0(y)|\leq C_0|y|^{-1/2},|B_1(y)|\leq C_1|y|^{-1/2},\]
		where $C_0,C_1>0$ are both constants independent of $y$.
	\end{lemma}
	We omit the detailed proof here, which is a natural extension of Lemma 2.4 in \cite{liImagingBuriedObstacles2021} based on the decaying property of oscillatory integrals \cite{steinHarmonicAnalysisRealVariable1993}, namely, the stationary phase lemma.
	\subsection{Global focusing}
	In this subsection, we give the global focusing properties of imaging general buried impenetrable obstacles through the time reversal operator. 
	
	\begin{theorem}\label{pglobalf}
		Assume $D$ is a sound-soft obstacle and let $\lambda\ne 0$ be an eigenvalue of ${T}$ and ${f}\in  L^2(\ps)$ be an eigenfunction of ${T}$ associated with $\lambda$.  
		Then the transmission field ${u}_{f}^t$, due to the Herglotz incident wave ${u}_{f}^i$ given by equation \eqref{herglotz}, has the following form
		\begin{eqnarray}\label{pag}
			{u}_f^{trans}(x) =   \int_{\Gamma_D}B_0(x-y)h_1(y)ds_y,
		\end{eqnarray}
		for a density function $h_1$ given in 
		the form of equation \eqref{hh1}.
		% , where
		% \begin{align}\label{jmth}
		% 	\widehat{\jmath}(k|x-y|)=\int_{\ps}Q^2(\alpha)\me^{\mi k_-\mv(\alpha)\cdot (x-y)}ds_\alpha
		% \end{align}
	\end{theorem}
	\begin{proof}
		Formula \eqref{pag} will be proved if we can write ${f}$ in the form
		\begin{align}\label{pff}
			{f}(\beta)=\int_{\Gamma_D}Q(\beta) \me^{-\mi k_-\mv(\beta)\cdot y}h_1(y)ds_y,
		\end{align}
		for a given density $h_1$. Indeed, if such a relation holds, then
		\begin{align*}
			{u}_f^{trans}(x)=\int_{\ps}{f}(\beta)Q(\beta)\me^{\mi k_-\mv(\beta)\cdot x}ds_\beta
			=\int_{\Gamma_D}h_1(y)\int_{\ps}Q^2(\beta)\me^{\mi k_-\mv(\beta)\cdot (x-y)}ds_\beta ds_y.
		\end{align*}
		To prove \eqref{pff}, we first note that for all $\beta\in\ps$, 
		\begin{align}\label{expf}
			{f}(\beta)=\frac{1}{\lambda}{T}{f}(\beta)=\frac{1}{\lambda}\int_{\ps}{t}(\beta,\alpha){f}(\alpha)ds_\alpha.
		\end{align}
		Given the reciprocity relation in Proposition \ref{farree},
		equation \eqref{kne} can be written as
		\begin{align*}
			{t}(\beta,\alpha)=\int_{-\ps}A(\gamma,\alpha)\overline{A(-\beta,-\gamma)}ds_\gamma.
		\end{align*}
		Using Dirichlet boundary condition in the integral representation \eqref{farre}, we get that
		\begin{align}\label{kne3}
			{t}(\alpha,\beta)=\int_{\Gamma_D}h^\alpha(y)Q(\beta)\me^{-\mi k_-\mv(\beta)\cdot y}ds_y,
		\end{align}
		where the density $h^\alpha$ is given by
		\begin{align*}
			h^\alpha(x)=\int_{-\ps}Q(\gamma)
			\int_{\Gamma_D}{\me^{-\mi k_-\mv(\gamma)\cdot y}}\frac{\partial u_{tot}^\alpha(y)}{\partial \nu(y)} ds_y\frac{\partial \overline{u_{tot}^{-\gamma}(x)}}{\partial \nu(x)}ds_\gamma.
		\end{align*}
		Here, $u_{tot}^\alpha$ represents the total field by an incident plane wave of direction $\alpha$.
		Combining equations \eqref{expf} and \eqref{kne3}, equation \eqref{pff} follows by setting the density 
		\begin{align}\label{hh1}
			h_1(x)=\frac{1}{\lambda}\int_{\ps}h^\alpha(x){f}(\alpha)ds_\alpha.
		\end{align}
	\end{proof}
	
	Following a similar argument, we can deduce the global focusing property for the sound-hard obstacles.
	\begin{theorem}\label{phglobalf}
		Assume $D$ is a sound-hard obstacle and let $\lambda\ne 0$ be an eigenvalue of ${T}$ and ${f}\in  L^2(\ps)$ be an eigenfunction of ${T}$ associated with $\lambda$. Then the transmission field ${u}_{f}^{trans}$, due to the Herglotz incident wave ${u}_{f}^i$ given by equation \eqref{herglotz}, has the following form
		\begin{eqnarray}\label{phag}
			{u}_f^{trans}(x) =   \int_{\Gamma_D}\nu(y)\cdot B_1(x-y)h_2(y)ds_y,
		\end{eqnarray}
		where $h_2(x)=\frac{1}{\lambda}\int_{\ps}h^\alpha(x){f}(\alpha)ds_\alpha$
        with
        	\begin{align*}
			h^\alpha(x)=\int_{-\ps}Q(\gamma)
			\int_{\Gamma_D}\frac{\partial \me^{-\mi k_-\mv(\gamma)\cdot y}}{{\partial \nu(y)}}{ u_{tot}^\alpha(y)} ds_y{ \overline{u_{tot}^{-\gamma}(x)}}ds_\gamma.
		\end{align*}
		% , where
		% \begin{align}\label{jhmth}
		% 	{\jmath}(k|x-y|)=\int_{\ps}\mv(\alpha)Q^2(\alpha)\me^{\mi k_-\mv(\alpha)\cdot (x-y)}ds_\alpha
		% \end{align}
	\end{theorem}

	Based on the asymptotic property in Lemma \ref{decay}, equations \eqref{pag}
	and \eqref{phag} show that the incident wave $\uincf$ generated by the eigenfunction $f$ will focus on the unknown obstacles and decay as $1/\sqrt{r}$ where $r$ is the distance from the boundaries of the obstacles. This is the so-called global focusing property, which is similar to the one in homogeneous media  \cite{hazardSelectiveAcousticFocusing2004}.
	In the following section, we will see that for small and well-resolved particles, a more surprising result named selective focusing also holds in layered media.

	\section{Selective focusing of multiple buried particles}\label{secle}
	In this section, we show that when the particles are small and well-resolved so that multiple scattering is weak, the number of particles contained in $\lowsp$ is related to the number of significant eigenvalues of the time reversal operator $T$, while the associated eigenfunctions can be used to selectively focus on the corresponding unknown particles. To prove this result, we first derive the leading time reversal operator of general-shaped particles in the asymptotic region, then compute its approximate dominant eigensystem, and finally show that the eigenfunctions corresponding to those significant eigenvalues can be used to achieve selective focusing.
	
	\subsection{Asymptotic behavior of  far field pattern}
	We consider the scattering by multiple sparsely distributed small particles by a scaling technique and asymptotic analysis. Let $D_l$, $1\leq l\leq M$, be a collection of simply connected reference scatterers containing the origin with $C^2$ boundary, and the diameter of each $D_l$ is comparable with the wavelength. 
    Denote multiple small scatterers for our reconstruction by
	\begin{align}\label{mulobs0}	D_\rho:=\bigcup\limits_{l=1}^MD_{\rho,l},\quad D_{\rho,l}:=s_l+\rho D_l,\quad s_l\in\lowsp,1\leq l\leq M,
	\end{align}
    % Throughout this subsection, we assume that the wavenumber $k=\oo(1)$. Hence, the size of a scatterer can be characterized by its Euclidean diameter. 
	where $s_l$ denotes the location of the $l$th scatterer and $\rho$ is a small positive scaling factor.
    % in $\lowsp$ with $C^2$ boundary. 
    We further assume that 
	\begin{align}\label{mulobs1}
		L=\min_{l\neq l',1\leq l,l'\leq M} \textrm{dist}(\overline{D_{\rho,l}},\overline{D_{\rho,l'}})\gg 1.
	\end{align}
For $\varphi\in
		C(\partial D_l)$, define the single and double layer boundary operators on $\partial D_l$:
	\begin{align}
		(S_l\varphi)(x) & :=\int_{\partial D_l}\Phi(x,y)\varphi(y)ds_y,\quad x\in\partial D_l,\label{sp}                                                                \\
		(K_l\varphi)(x) & :=\int_{\partial D_l}\frac{\partial \Phi(x,y)}{\partial\nu(y)}\varphi(y)ds_y,\quad x\in\partial D_l,\label{dp}
	\end{align}
	where $\Phi(x,y)$ is the layered Green's function defined by \eqref{green}. The
	adjoint operator $K_l'$ of $K_l$ is given by
	\begin{align}
		(K_l'\varphi)(x):=\int_{\partial D_l}\frac{\partial \Phi(x,y)}{\partial\nu(x)}\varphi(y)ds_y,\quad x\in\partial D_l.\label{asp}
	\end{align}
	It holds the series expansions for $\Phi(x,y)$  \cite{coltonInverseAcousticElectromagnetic2019}:
	\begin{align}\label{teshu}
		\Phi(x,y)=& \Phi_0(x,y)+\Phi^{ref}(x,y)\notag\\
        &+\left(\frac{\mi}{4}-\frac{1}{2\pi}\ln\frac{k}{2}-\frac{C_E}{2\pi}\right)
		+\mathcal{O}({|x-y|^2\ln(|x-y|)}),\quad x\neq y,\\
		\nabla_y\Phi(x,y)=&\nabla_y\Phi_0(x,y)+\nabla_y\Phi^{ref}(x,y)
		+\mathcal{O}({|x-y|\ln(|x-y|)}),\quad x\neq y,\notag
	\end{align}
	where $\Phi_0(x,y):=-(1/2\pi)\ln|x-y|$ denotes the fundamental solution for the Laplace equation, $C_E = 0.577 215 66...$ is Euler’s constant, and $\Phi^{ref}(\cdot,y)$ is an analytic function in $\mr_{\pm}$.
	
	Similar to the definitions of $S_l, K_l, K_l'$, we define the operators ${S}_l^0, {K}_l^0, {K}_l^{0,'}$ in the same way as equations \eqref{sp}, \eqref{dp}
	and \eqref{asp}, but with $\Phi(x,y)$ replaced by $\Phi_0(x,y)$.

	The following lemma indicates that the multiple scattering effect is weak under condition \eqref{mulobs1} \cite{liRecoveringMultiscaleBuried2015}.
	\begin{lemma}\label{nelmul}
		Consider multiple buried particles $D_\rho$ in \eqref{mulobs0} under the condition \eqref{mulobs1}. Then we have
		\[u^\infty(\hat{x};\bigcup\limits_{l=1}^MD_{\rho,l})=\sum_{l=1}^M u^\infty(\hat{x};D_{\rho,l})+\oo\left(\frac{1}{\sqrt{L}}\right).\]
	\end{lemma}
    
% Before our asymptotic analysis for the far field, we give the following lemma that will be further used.

	Now we are in a position to present the conclusion on scattering from multiple small buried scatterers mainly using the technique of integral equations\cite{griesmaierMultifrequencyOrthogonalitySampling2011,liRecoveringMultiscaleBuried2015}. It is noted here that a similar asymptotic formula can be obtained based on  Foldy-Lax model \cite{challaJustificationFoldyLaxApproximation2014}  when the multiple scattering effect is taken into account.
	\begin{theorem}
		Consider multiple buried particles $D_\rho$ in \eqref{mulobs0}.
		Let $	u^\infty(\hat{x};D_\rho)\in L^2(\ps)$ be the far field pattern  excited by a plane wave $\uinc(x)=\me^{\mi k_+ x\cdot d}$. Then, for sufficiently large $L$ and $\rho\rightarrow +0$, the far field pattern corresponding to the sound-soft case satisfies
		\begin{align}\label{farps}
			\begin{aligned}
				u^\infty(\hat{x};D_\rho) = 2\pi ({\ln\rho})^{-1} Q(\hat{x})Q(d)\sum_{l=1}^{M}\me^{-\mi k_-(\mv(\hat{x})-\mv(d))\cdot s_l}
			+\mathcal{O}\left((\ln\rho)^{-2}+L^{-1/2}\right),
			\end{aligned}
		\end{align}
		and the far field pattern corresponding to the sound-hard case satisfies
		\begin{align}\label{asyfar5}
			\begin{aligned}
				u^\infty(\hat{x};D_\rho)=&
				- (k_-)^2\rho^2Q(\hat{x})Q(d)\sum_{l=1}^{M}\left(|D_l|+\mv(\hat{x})\cdot\mathbb{M}_l\mv(d)\right)\me^{-\mi k_-(\mv(\hat{x})-\mv(d))\cdot s_l}+\\
				&\mathcal{O}\left(\rho^3\ln \rho+L^{-1/2}\right),
			\end{aligned}
		\end{align}
		where $|D_l|$ denotes the area of $D_l$ and $\mathbb{M}_l$ given in \eqref{tendef} denotes the polarization tensor corresponding to the sound-hard scatterer $D_l$.
	\end{theorem}
	\begin{proof}
		According to Lemma \ref{nelmul}, with an error of order $\oo(1/\sqrt{L})$, it suffices to give the asymptotics of the far field pattern for a single component of $D_\rho$ in \eqref{mulobs0}. Without loss of generality, we suppose $D_\rho = z+\rho D$ where $D$ is one of the reference scatterers. 
		
		We first discuss the sound-soft case. For
		$f\in C(\partial {D_\rho})$ and $g\in C(\partial {{D}})$, let us introduce the transforms
		\begin{align*}
			\hat{f}(\xi)=f^\wedge:=f(\rho\xi+z),\xi\in\partial {{D}},\quad
			\check{g}(x)=g^\vee:=g((x-z)/\rho),x\in\partial {D_\rho}.
		\end{align*}	
		We represent the scattered field $\usc(x;{D_\rho})$ as the combined layer potential
		\begin{align*}
			u^{sc}(x;D_\rho)=\int_{\partial{D_\rho}}\left\{\frac{\partial\Phi(x,y)}{\partial\nu(y)}-\mi \varrho\Phi(x,y)\right\}\varphi(y)d s_y,
			\quad x\in\mr\setminus\overline{{D_\rho}},
		\end{align*}
		where the coupling parameter $\varrho$ is chosen to be
		\[\varrho:=(\rho\ln\rho)^{-1}.\]
		By using the homogeneous Dirichlet boundary condition and the jump properties of the integral operators, we see that
		the density function $\varphi\in C(\partial {D_\rho})$ satisfies
		\begin{align*}
			\left(\frac12I+K_{D_\rho}-\mi\varrho S_{D_\rho}\right)\varphi=-{u}^{trans}|_{\partial{D_\rho}},
		\end{align*}
		where $I$ denotes the identity operator. It is noted that the operator $\frac12I+K_{D_\rho}-\mi\varrho S_{D_\rho}:C(\partial{D_\rho})\rightarrow C(\partial{D_\rho})$ is bijective \cite{kressLinearIntegralEquations2014}.
		Hence, using the asymptotic behavior of the Green's function, the far field pattern $u^\infty$ is given by
		\begin{align}\label{far}
			\begin{aligned}
				&u^\infty(\hat{x};{D_\rho}) \\
                =& - Q(\hat{x})\int_{\partial{D_\rho}}\left(\frac{\partial\me^{-\mi k_-\mv(\hat{x})\cdot y}}{\partial\nu(y)}-\mi\varrho\me^{-\mi k_-\mv(\hat{x})\cdot y}\right)
				\left[\left(\frac12I+K_{D_\rho}-\mi\varrho S_{D_\rho}\right)^{-1}\psi\right](y)d s_y\\
				 =&\mi Q(\hat{x})\int_{\partial{D_\rho}}\left(k_-\mv(\hat{x})\cdot\nu(y) +\varrho\right)\me^{-\mi k_-\mv(\hat{x})\cdot y}
				\left[\left(\frac12I+K_{D_\rho}-\mi\varrho S_{D_\rho}\right)^{-1}\psi\right](y)d s_y,
			\end{aligned}
		\end{align}
		where $\psi:= {u}^{trans}|_{\partial{D_\rho}}$.
        
       Now we give an estimate of $\left(\frac12I+K_{D_\rho}-\mi\varrho S_{D_\rho}\right)^{-1}\psi$. For $\phi\in C(D_\rho)$,
        define $M^0_{D}: C(\partial{D})\rightarrow C(\partial{D})$ by
		\[(M^0_{D} \phi)(x)=\frac{1}{2\pi}\int_{\partial{D}}\phi(y)ds_y,\quad x\in\partial{D}.\] 
        Using change of variables in the integrals, we have that
		\begin{align}\label{ss1}
			\begin{aligned}
				(S_{D_\rho}^0\phi)(x)&=\int_{\partial D_\rho} \Phi_0(x,y)\phi(y)ds_y
				=-\frac{1}{2\pi}\int_{\partial D_\rho} \ln(|x-y|)\phi(y)ds_y
				\\
				&=-\frac{1}{2\pi}\int_{\partial D} \ln(\rho |\xi-\eta|)\hat{\phi}(\eta)\rho ds_\eta=-\rho\ln\rho (M_D^0 \hat{\phi})(\xi)+\oo(\rho),
			\end{aligned}
		\end{align}
		\begin{align*}
			\begin{aligned}
				\left(K_{D_\rho}^0\phi\right)(x)&=\int_{\partial {D_\rho}}\frac{\partial \Phi_0(x,y)}{\partial\nu(y)}\phi(y)ds_y=\frac1{4\pi}\int_{\partial{D_\rho}}\frac{\nu(y)\cdot(x-y)}{|x-y|^2}\phi(y)ds_y \\
				&=\frac1{4\pi}\int_{\partial D}\frac{\nu{\left(\eta\right)}\cdot{\left(\xi-\eta\right)}}{\rho\left|\xi-\eta\right|^2}\hat{\phi}{\left(\eta\right)}\rho ds_\eta=\left(K_D^0\hat{\phi}\right)(\xi).
			\end{aligned}
		\end{align*}
		With the expansion in equation \eqref{teshu}, we obtain that as $\rho\rightarrow 0$,
		\begin{align}
			\begin{aligned}
				\left(\left(S_{D_\rho}-S_{D_\rho}^{0}\right)\phi\right)(x)& =\int_{\partial D_\rho}\left[\Phi(x,y)-\Phi_0(x,y)\right]\phi(y)ds_y  \\
				&=\int_{\partial D}\mathcal{O}(1)\hat{\phi}\left(\eta\right)\rho ds_\eta=\mathcal{O}\left(\rho\right),
			\end{aligned}
		\end{align}
        and
		\begin{align}\label{ss4}
			\begin{aligned}
				\left(\left(K_{D_\rho}-K_{D_\rho}^0\right)\phi\right)(x)& =\int_{\partial{D_\rho}}\left(\frac{\partial\left[\Phi(x,y)-\Phi_0(x,y)\right]}{\nu(y)}\right)\phi(y)ds_y  \\
				&=\int_{\partial D}\mathcal{O}(1)\hat{\phi}{\left(\eta\right)}\rho ds_\eta=\mathcal{O}\left(\rho\right).
			\end{aligned}
		\end{align}
        Thus,
		\begin{align}\label{opas}
			\left(\frac12I+K_{D_\rho}-\mathrm{i}\varrho S_{D_\rho}\right)\phi=\left(\left(\frac12I+K_D^0+\mathrm{i}M_D^0\right)\hat{\phi}\right)^\vee+\mathcal{O}\biggl(\frac1{\ln\rho}\biggr),
		\end{align}
		uniformly on compact subsets of $C(\partial D_\rho)$ as $k_-\rho \rightarrow 0$. Since  
		$\frac12I+K_D^0+\mathrm{i}M_D^0$ is  also bijective and its inverse is bounded \cite{kressLinearIntegralEquations2014},
		% Changing the variable $y=z+\rho \xi$ with $\xi\in\partial {{D}}$ in \eqref{far},
		making use of the estimate \eqref{opas} and Theorem 10.1 in \cite{kressLinearIntegralEquations2014}, we find that
		\begin{align}\label{asyfa}
			\begin{aligned}
				&u^\infty(\hat{x};{D_\rho})\\
                =&\mi\rho Q(\hat{x})\int_{\partial {{D}}}\left(k_-\mv(\hat{x})\cdot\nu(y) +\varrho\right)\me^{-\mi k_-\mv(\hat{x})\cdot (z+\rho\xi)}
\left[\left(\frac12I+K_D^0+\mathrm{i}M_D^0\right)^{-1}\hat{\psi}(\xi)\right. \\ &\left.+\mathcal{O}\left(\frac{1}{\ln\rho}\right)\right]d s_\xi.
			\end{aligned}
		\end{align}
		Expanding the exponential function $\xi\rightarrow\me^{-\mi k_-\mv(\hat{x})\cdot(z+\rho\xi)}$ around $z$ in terms of $\rho$ yields
		\begin{align}\label{incex}
			\me^{-\mi k_-\mv(\hat{x})\cdot(z+\rho\xi)}=\me^{-\mi k_-\mv(\hat{x})\cdot z}(1-\mi k_-\mv(\hat{x})\cdot\xi\rho+\mathcal{O}(\rho^2)),\quad \rho\rightarrow0.
		\end{align}
		Inserting \eqref{incex} into \eqref{asyfa} gives
		\begin{align}\label{asyfaii}
			\begin{aligned}
				&u^\infty(\hat{x};{D_\rho})\\
				 =&\mi(\ln\rho)^{-1} Q(\hat{x})\me^{-\mi k_-\mv(\hat{x})\cdot z}\int_{\partial {{D}}}
				\left(\frac12I+K_D^0+\mathrm{i}M_D^0\right)^{-1}\hat{\psi}(\xi)d s_\xi +\oo((\ln\rho)^{-2}).
			\end{aligned}
		\end{align}
		To estimate the integrals on the right hand side of \eqref{asyfaii}, we will investigate the asymptotics for the transmitted wave ${u}^{trans}=Q(d)\me^{\mi k_-\mv(d)\cdot x}$ corresponding to the plane wave $\uinc=\me^{\mi kx\cdot d}$. Expanding the function
		$\xi\rightarrow({u}^{trans})^\wedge(\xi)= {u}^{trans}(\rho\xi+z)$, we have
		\begin{align}\label{normexp}
			({u}^{trans})^\wedge(\xi)=Q(d)\me^{\mi k_-\mv(d)\cdot(z+\rho\xi)}=Q(d)\me^{\mi k_-\mv(d)\cdot z}(1+\mathcal{O}(\rho)),\quad \rho\rightarrow0.
		\end{align}
		Then, it holds
		\begin{align}\label{firstter}
			\begin{aligned}
				&\int_{\partial {{D}}}\left(\frac12I+K_D^0+\mathrm{i}M_D^0\right)^{-1}\hat{\psi}(\xi)d s_\xi\\
				 =&Q(d)\me^{\mi k_-\mv(d)\cdot z}\int_{\partial {{D}}}\left(\frac12I+K_D^0+\mathrm{i}M_D^0\right)^{-1}\mathbbm{1}d s_\xi+\mathcal{O}(\rho),
			\end{aligned}
		\end{align}
		where $\mathbbm{1}$ is the constant function with value 1 on $\partial D$. Inserting \eqref{firstter} into \eqref{asyfaii} yields
		\begin{align}\label{final1}
			\begin{aligned}
				u^\infty(\hat{x};{D_\rho})=
				& \mi(\ln\rho)^{-1} Q(\hat{x})Q(d)\me^{-\mi k_-(\mv(\hat{x})-\mv(d))\cdot z}\int_{\partial {{D}}}\left(\frac12I+K_D^0+\mathrm{i}M_D^0\right)^{-1}\mathbbm{1}d s_\xi
				\\
				& +\oo((\ln\rho)^{-2}).
			\end{aligned}
		\end{align}
        Consider the integral equation
        \begin{align}\label{inttt}
        \left(\frac12I+K_D^0+\mathrm{i}M_D^0\right)\phi=\mathbbm{1},
        \end{align}
		Recalling that \cite{kressLinearIntegralEquations2014}
		\begin{align}\label{iddd}
			{K}_{{D}}^0\mathbbm{1}=\int_{\partial {{D}}}\frac{\partial \Phi_0(x,y)}{\partial\nu(y)}d s_y=-\frac{\mathbbm{1}}{2},
            \end{align}
        which implies $\left(\frac{I}{2}+K^{0}_{D}\right)\mathbbm{1}=0$. Thus, inserting $\phi=c\mathbbm{1}$ into \eqref{inttt}, we have $c=\frac{2\pi}{\mi |D|}$, i.e.,
        \begin{align}\label{rrr}
            \int_{\partial {{D}}}\left(\frac12I+K_D^0+\mathrm{i}M_D^0\right)^{-1}\mathbbm{1}d s_\xi = \frac{2\pi}{\mi}.
        \end{align}
		Combining \eqref{final1} and  \eqref{rrr}, we obtain that
		\begin{align*}
			\begin{aligned}
				u^\infty(\hat{x};{D_\rho})=
				& 2\pi(\ln\rho)^{-1} Q(\hat{x})Q(d)\me^{-\mi k_-(\mv(\hat{x})-\mv(d))\cdot z}+\oo((\ln\rho)^{-2}).
			\end{aligned}
		\end{align*}
		
		For the sound-hard case, we represent the scattered field $\usc(x;{D_\rho})$ as the single layer potential
		\begin{align*}
			u^{sc}(x;{D_\rho})=\int_{\partial{D_\rho}}\Phi(x,y)\varphi(y)d s_y,
			\quad x\in\mr\setminus\overline{{D_\rho}},
		\end{align*}
		with the density function $\varphi\in C(\partial {D_\rho})$ given by
		\begin{align*}
			\varphi=\left(\frac{I}{2}-K'_{D_\rho}\right)^{-1}(\partial_\nu {u}^{trans}|_{\partial{D_\rho}}).
		\end{align*}
		Then, using the asymptotic behavior of the Green's function, the far field pattern $u^\infty$ is given by
		\begin{align}\label{far22}
			u^\infty(\hat{x};{D_\rho}) & = Q(\hat{x})\int_{\partial{D_\rho}}\me^{-\mi k_-\mv(\hat{x})\cdot y}
			\left[\left(\frac{I}{2}-K'_{D_\rho}\right)^{-1}\psi\right](y)d s_y,
		\end{align}
		where $\psi:=\partial_\nu {u}^{trans}|_{\partial{D_\rho}}=\nu\cdot\nabla {u}^{trans}|_{\partial{D_\rho}}$.
		Changing the variable $y=z+\rho \xi$ with $\xi\in\partial {{D}}$ in \eqref{far22} and
		making use of the estimate in Lemma 2.4 in  \cite{liRecoveringMultiscaleBuried2015}, we find that
		\begin{align}\label{asyfar}
			u^\infty(\hat{x};{D_\rho}) & =\rho Q(\hat{x})\int_{\partial {{D}}}\me^{-\mi k_-\mv(\hat{x})\cdot (z+\rho\xi)}
			\left[\left(\frac{I}{2}-K^{0,'}_{D}\right)^{-1}\hat{\psi}(\xi)+\oo(\rho^2\ln\rho)\right]d s_\xi.
		\end{align}
		Expanding the exponential function $\xi\rightarrow\me^{-\mi k_-\mv(\hat{x})\cdot(z+\rho\xi)}$ around $z$ in terms of $\rho$ yields
		\begin{align}\label{incexp}
			\me^{-\mi k_-\mv(\hat{x})\cdot(z+\rho\xi)}=\me^{-\mi k_-\mv(\hat{x})\cdot z}-\mi k_-\rho(\mv(\hat{x})\cdot\xi)\me^{-\mi k_-\mv(\hat{x})\cdot z}+\oo(\rho^2),\quad \rho\rightarrow0.
		\end{align}
		Inserting \eqref{incexp} into \eqref{asyfar} gives
		\begin{align}\label{asyfar2}
			\begin{aligned}
				u^\infty(\hat{x};{D_\rho})=
				& \rho Q(\hat{x})\me^{-\mi k_-\mv(\hat{x})\cdot z}\left(\int_{\partial {{D}}}
				\left(\frac{I}{2}-K^{0,'}_{D}\right)^{-1}\hat{\psi}(\xi)d s_\xi\right)
				\\
				& -\mi k_-\rho^2 Q(\hat{x})\me^{-\mi k_-\mv(\hat{x})\cdot z}\left(\int_{\partial {{D}}}(\mv(\hat{x})\cdot\xi)
				\left(\frac{I}{2}-K^{0,'}_{D}\right)^{-1}\hat{\psi}(\xi)d s_\xi\right)
				\\
				& +\oo(\rho^3\ln \rho).
			\end{aligned}
		\end{align}
		To estimate the integrals on the right hand side of \eqref{asyfar2}, we will investigate the asymptotics for the transmitted wave corresponding to plane wave $\uinc=\me^{\mi k_+d\cdot x}$. Expanding the function
		$\xi\rightarrow(\nabla {u}^{trans})^\wedge(\xi)=\nabla {u}^{trans}(\rho\xi+z)$ yields that
		\begin{align}\label{normexp2}
			(\nabla {u}^{trans})^\wedge(\xi)=Q(d)\mi k_-\mv(d)\me^{\mi k_- \mv(d)\cdot z}
			[ 1+\mi k_-(\mv(d) \cdot\xi)\rho+\oo(\rho^2) ],\quad\rho\rightarrow 0.
		\end{align}
		% Recalling that \cite{kressLinearIntegralEquations2014}
		% \begin{align*}
		% 	{K}_{{D}}^0\mathbbm{1}=\int_{\partial {{D}}}\frac{\partial \Phi_0(x,y)}{\partial\nu(y)}d s_y=-\frac{1}{2},
		% \end{align*}
		From \eqref{iddd}, we see that $\left(\frac{I}{2}-K^{0}_{D}\right)^{-1}\mathbbm{1}=\mathbbm{1}$, and thus
		\begin{align}\label{firstterm}
			\begin{aligned}
				\int_{\partial {{D}}}\left(\frac{I}{2}-K^{0,'}_{D}\right)^{-1}\hat{\psi}(\xi)d s_\xi
				& =\int_{\partial {{D}}}\hat{\psi}(\xi)\left(\frac{I}{2}-K^{0}_{D}\right)^{-1}\mathbbm{1}d s_\xi
				\\
				& =\int_{\partial {{D}}}\nu(\xi)\cdot(\nabla {u}^{trans})^\wedge(\xi)d s_\xi.
			\end{aligned}
		\end{align}
		Inserting \eqref{normexp2} into \eqref{firstterm} and applying Gauss's theorem yields
		\begin{align}\label{term1}
			\begin{aligned}
				& \int_{\partial {{D}}}\left(\frac{I}{2}-K^{0,'}_{D}\right)^{-1}\hat{\psi}(\xi)d s_\xi
				\\
				& =\int_{\partial {{D}}}\nu(\xi)\cdot(Q(d)\mi k_-\mv(d)\me^{\mi k_- \mv(d)\cdot z}
				[ 1+\mi k_-(\mv(d) \cdot\xi)\rho+\oo(\rho^2) ])d s_\xi
				\\
				& =-Q(d)(k_-)^2\rho|{{D}}|\me^{\mi k_- \mv(d)\cdot z }+\oo(\rho^2),
			\end{aligned}
		\end{align}
        %$|{{D}}|$ denotes the area of ${{D}}$ and
		where  the last equality follows from the identity
		\[\textrm{div}_\xi \mv(d)(\mv(d)\cdot\xi)=1.\]
		Again using \eqref{normexp} we can evaluate the second integral over $\partial {{D}}$ on the right hand of \eqref{asyfar2} as follows,
		\begin{align}\label{term2}
			\begin{aligned}
				& \int_{\partial {{D}}}(\mv(\hat{x})\cdot\xi)\left(\frac{I}{2}-K^{0,'}_{D}\right)^{-1}\hat{\psi}(\xi)d s_\xi
				\\=
				& \int_{\partial {{D}}}(\mv{(\hat{x})}\cdot\xi)
				\left(\frac{I}{2}-K^{0,'}_{D}\right)^{-1}(\nu(\xi)\cdot \mi k_-Q(d)\mv(d)\me^{\mi k_-\mv(d) \cdot z })d s_\xi+\oo(\rho)
				\\=
				& -\mi k_-Q(d)\me^{\mi k_-\mv(d)\cdot z}(\mv(\hat{x})\cdot\mathbb{M}\mv(d))+\oo(\rho),
			\end{aligned}
		\end{align}
		where the polarization tensor $\mm$ is defined as
		\begin{align}\label{tendef}
			\mathbb{M}=-\int_{\partial {{D}}}\xi\otimes\left(\frac{I}{2}-K^{0,'}_{D}\right)^{-1}\nu(\xi)
			d s_\xi.
		\end{align}
		Now, combining \eqref{asyfar2}, \eqref{term1} and \eqref{term2} gives the asymptotics \eqref{asyfar5}, namely,
		\begin{align*}
			\begin{aligned}
				u^\infty(\hat{x};{D_\rho})=
				& - (k_-)^2\rho^2Q(d)Q(\hat{x})\left(|{{D}}|+\mv(\hat{x})\cdot\mathbb{M}\mv(d)\right)\me^{-\mi k_-(\mv(\hat{x})-\mv(d))\cdot z}\\&+\oo(\rho^3\ln \rho).
			\end{aligned}
		\end{align*}
	\end{proof}

    	\begin{remark}\label{prop}
		Here the polarization tensor $\mathbb{M}$ can be considered as a $2\times 2$ symmetric
		negative definite matrix \cite{rammWaveScatteringSmall2005}
		\begin{align*}
			\mathbb{M}_{ij}=-\int_{\mr\setminus D}\nabla\phi_i\nabla\phi_jd\xi-\delta_{ij}|D|,
		\end{align*}
		where  $\phi_j$ is the solution to the following exterior Neumann problem
		\[\begin{cases*}
			\Delta \phi_j=0,\\
			\frac{\partial \phi_j}{\partial \nu}=-\nu_j~~\textrm{on}~\partial 
			D,\\
			\phi_j=o(1).
		\end{cases*}\]
		In the case of a sound-hard disk, the polarization tensor becomes: $\mm=-2|D|I$.
	\end{remark}
	\subsection{Selective focusing of sound-soft particles}
    From expansion \eqref{farps}, the leading far field pattern $A^0$ for sound-hard particles by
	excluding the scaling factor and remainder terms is given by
	\begin{align*}
		A^0(\beta,\alpha)= Q(\beta)Q({\alpha})\sum_{j=1}^{M}\me^{-\mi k_{-}(\mathrm{v}(\beta)-\mv(\alpha))\cdot s_{j}},\mbox{ with }\beta\in-\ps,\alpha\in\ps.
	\end{align*}
	Consequently, the leading time reversal operator denoted ${T}^0\in\mathcal{L}(L^2(\ps))$ is the integral operator with kernel $${t}^0(\beta,\alpha)=\int_{-\ps}A^0(\gamma,\alpha)\overline{A^0(\gamma,\beta)}ds_\gamma.$$
	From the remainder terms of \eqref{farps}, it is easy to verify that
	\begin{align}\label{saei}
		\begin{aligned}
			\left\|\left(\frac{\ln\rho}{2\pi}\right)^2 T^\rho-T^0\right\|_{\mathcal{L}(L^2(\ps))}=\oo((\ln\rho)^{-1}+L^{-1/2}),
		\end{aligned}
	\end{align}
	where $T^\rho$ is the original time reversal operator. Since $T^\rho$ is compact and self-adjoint, perturbation theory \cite{katoPerturbationTheoryLinear1995} ascertains the continuity of any finite system of eigenvalues as well as of the associated eigen-projection. Therefore, to get the approximate eigensystem of the time reversal operator, we can study the leading time reversal operator $T^0$ instead of the original $T^\rho$~ \cite{hazardSelectiveAcousticFocusing2004}.
	
	Given ${f}\in L^2(\ps)$, the leading time reversal operator $T^0$ can be represented by
	\begin{align}\label{t0re}
		\begin{aligned}
			{T}^0{f}(\beta)&=\int_{\ps}{t}^0(\beta,\alpha){f}(\alpha)ds_\alpha\\
            &=\int_{\ps}\int_{-\ps}A^0(\gamma,\alpha)\overline{A^0(\gamma,\beta)}ds_\gamma{f}(\alpha)ds_\alpha\\
			&=\int_{-\ps} Q^2(\gamma)ds_\gamma\sum_{p=1}^{M}\int_{\ps}Q({\alpha})\overline{e_p(\alpha)}{f}(\alpha)ds_\alpha Q({\beta}) e_p(\beta)+ \mathcal{O}((k_-L)^{-1/2}),
		\end{aligned}
	\end{align}
	where \[e_l(\alpha)=\me^{-\mi k_-\mv(\alpha)\cdot s_l},\quad\forall \alpha\in\ps, l=1,\dots,M,\]
	and the interaction terms are contained in the residual based on Lemma \ref{decay}, i.e.,
	\[\int_{-\ps} Q^2(\gamma)\me^{-\mi k_{-}\mathrm{v}(\gamma)\cdot ({s_p-s_{q}}})ds_\gamma=\mathcal{O}((k_-|s_p-s_{q}|)^{-1/2}),~p\neq q.\]
	
	Let us introduce functions $g_l(\alpha),{h}_{l,j}(\alpha), l=1,\cdots,M, j=1,2,$ defined by
	\begin{align}\label{gl}
		{g_l}(\alpha)=Q(\alpha){e}_l(\alpha),\qquad
		{h}_{l,j}(\alpha)=\boldsymbol{\mathrm{e}}_j\cdot \mv(\alpha) Q(\alpha){e}_l(\alpha),\quad\forall \alpha\in\ps,
	\end{align}
	where $\boldsymbol{\mathrm{e}}_j$ denotes the basis in $\mr, j=1,2$.
	% \begin{remark}\label{rich}
		We note that $g_l(\alpha)$ and $h_{l,q}(\alpha)$ are exactly the far fields corresponding to monopole and dipole sources located at the unknown particles in $\lowsp$, respectively. More precisely, $g_l(\alpha)=Q(\alpha)\me^{-\mi k_-\mv(\alpha)\cdot s_l}$ is the far field in the direction $\alpha\in\ps$ due to the point source $\Phi(\cdot,s_l)$ located at $s_l$ in $\lowsp$, while
		$h_{l,q}(\alpha)=(\boldsymbol{\mathrm{e}}_q\cdot\mv(\alpha) )Q(\alpha)\me^{-\mi k_-\mv(\alpha)\cdot s_l}$ represents the far field in the direction $\alpha\in\ps$
		due to the dipole source $\frac{1}{\mi k_-}\nabla \Phi(\cdot,s_l)\cdot \boldsymbol{\mathrm{e}}_q$ of
		direction $\boldsymbol{\mathrm{e}}_q$ located at $s_l$ in $\lowsp$.
	% \end{remark}
	
	\begin{theorem}
		The leading time reversal operator $T^0$ is a positive self-adjoint operator with finite rank $M$. Furthermore, when $k_-L\rightarrow\infty$, the family $\{g_l: l=1,\dots,M\}$ 
		constitutes an approximate basis of eigenfunctions corresponding to nonzero eigenvalues of ${T}^0$ with an error of order $\mathcal{O}((k_-L)^{-1/2})$.
	\end{theorem}
	\begin{proof}
		The bilinear form associated with $T^0$ is
		\[\left(T^0 f(\alpha),f'(\beta)\right)=\int_{-\ps}\int_{\ps}A^0(\gamma,\alpha)f(\alpha)ds_\alpha\int_{\ps}\overline{A^0(\gamma,\beta)f'(\beta)}ds_\beta ds_\gamma,\]
		which is clearly positive and self-adjoint.

		From equation \eqref{t0re}, we readily get that the range of $T^0$ is approximately spanned by $\{g_l: l=1,\dots,M\}$, which indicates the number of eigenfunctions corresponding to the nonzero eigenvalues of $T^0$ is at most $M$. 
		
		On the other hand, it holds
		\begin{align*}
			{T}^0{g_l}(\beta)&=\int_{-\ps} Q^2(\gamma)ds_\gamma\sum_{p=1}^{M}\int_{\ps}Q({\alpha})\overline{e_p(\alpha)}Q(\alpha){e}_l(\alpha)ds_\alpha Q({\beta}) e_p(\beta)+ \mathcal{O}((k_-L)^{-1/2})\\
			&=\int_{-\ps} Q^2(\gamma)ds_\gamma\int_{\ps} Q^2({\alpha})ds_\alpha Q({\beta}) e_l(\beta)+ \mathcal{O}((k_-L)^{-1/2})\\
			&={\lambda_l}{g_l}(\beta)+\mathcal{O}((k_-L)^{-1/2}),
		\end{align*}
		where ${\lambda_l}=\int_{-\ps} Q^2(\gamma)ds_\gamma\int_{\ps} Q^2({\alpha})ds_\alpha$. Combining the linear independence of $g_l(\alpha)$, $l=1,\dots,M$, which follows from Lemma \ref{indep},
		we prove that $\{g_l: l=1,\dots,M\}$ constitutes an approximate basis of eigenfunctions corresponding to the nonzero eigenvalues of ${T}^0$.
	\end{proof}
	\begin{theorem}
		For $1\leq l\leq M$, the incident Herglotz wave associated with the approximate eigenfunction ${g_l}$ selectively focuses on the $l$th particle.
		
	\end{theorem}
	\begin{proof}
		Due to the incident wave $u^i_{{g_l}}=\int_{\ps} Q(\alpha){e}_l(\alpha)\me^{i k_+\alpha\cdot x}ds_\alpha$, the transmitted wave can be expressed as follows
		\[{u}^{trans}_{{g_l}}=\int_{\ps} Q^2(\alpha)\me^{i k_-\mv(\alpha)\cdot (x-s_l)}ds_\alpha,\]
		which decays rapidly away from the $l$th particle as $\mathcal{O}((k_-|x-s_{l}|)^{-\frac{1}{2}})$.
	\end{proof}
	
	\subsection{Selective focusing of sound-hard particles}
	Similar to the sound-soft case, taking the leading term in \eqref{asyfar5}, we derive the leading far field pattern $A^0$ as
	\begin{align*}
		A^0(\beta,\alpha)= Q(\beta) Q({\alpha})\sum_{j=1}^{M}\me^{-\mi k_{-}(\mathrm{v}(\beta)-\mv(\alpha))\cdot s_{j}}\bigg(\mathrm{v}(\beta)\cdot \mm_j\mathrm{v}(\alpha)+|D_j|\bigg),\quad \beta\in-\ps,\alpha\in\ps.
	\end{align*}
	Given ${f}\in L^2(\ps)$, the leading time reversal operator can be represented by
	\begin{align}\label{stt}
		\begin{aligned}
			{T}^0{f}(\beta)&=\int_{\ps}{t}^0(\beta,\alpha){f}(\alpha)ds_\alpha=\int_{\ps}\int_{-\ps}A^0(\gamma,\alpha)\overline{A^0(\gamma,\beta)}ds_\gamma{f}(\alpha)ds_\alpha\\
			%		&=\int_{\ps}\int_{-\ps}\sum_{p=1}^{M} Q(\gamma) Q({\alpha})\bigg(\mathrm{v}(\gamma)\cdot \mm_p\mathrm{v}(\alpha)+|D_p|\bigg)\me^{-\mi k_{-}(\mathrm{v}(\gamma)-\mv(\alpha))\cdot s_{p}}\\
			%		&~~~~~~~~~~\sum_{q=1}^{M}         Q(\gamma) Q({\beta})\bigg(\mathrm{v}(\gamma)\cdot \mm_q\mathrm{v}(\beta)+|D_q|\bigg)\me^{\mi k_{-}(\mathrm{v}(\gamma)-\mv(\beta))\cdot s_{q}}ds_\gamma{f}(\alpha)ds_\alpha\\
			&=\sum_{q=1}^{M}\bigg\{\sum_{p=1}^{M}\int_{\ps}\Bigg[\int_{-\ps}\bigg( Q(\gamma) Q({\alpha})\big(\mathrm{v}(\gamma)\cdot \mm_p\mathrm{v}(\alpha)+|D_p|\big)\bigg)\\
			&~~~~\bigg( Q(\gamma) Q({\beta})\big(\mathrm{v}(\gamma)\cdot \mm_q\mathrm{v}(\beta)+|D_q|\big)\bigg)e_p(\gamma)\overline{e_q(\gamma)}ds_\gamma\Bigg]\overline{e_p(\alpha)}{f}(\alpha)ds_\alpha\bigg\} e_q(\beta),
		\end{aligned}
	\end{align}
	%where \[e_l(\alpha)=\me^{-\mi k_-\mv(\alpha)\cdot s_l},~l=1,\dots,M,\]
	Using the same conclusion as equation \eqref{saei}, we just need to find the eigensystem of ${T}^0$ in the sound-hard case. 
	\begin{theorem}
		When $k_-L\rightarrow\infty$, $\mathcal{A}_l$  given by \eqref{aaaa} with $1\leq l\leq M$ is an approximate invariant subspace for ${T}^0$. Furthermore, there exist three linearly independent approximate eigenfunctions $f_{l,q},q=1,2,3$ in $\mathcal{A}_l$, which are all linear combinations of ${g}_l$ and ${h}_{l,j}, j=1,2,$ with an error of order $\mathcal{O}((k_-L)^{-1/2})$.
	\end{theorem}
	\begin{proof}
		By utilizing  Lemma \ref{decay} to handle the interaction terms in \eqref{stt}, we obtain
		\begin{align*}
			{T}^0{f}(\beta)
			&=\sum_{q=1}^{M}\bigg\{(|D_q|)^2\int_{-\ps}
			Q^2(\gamma)ds_\gamma\int_{\ps} Q({\alpha})\overline{e_q(\alpha)}{f}(\alpha)ds_\alpha\bigg\}  Q({\beta})e_q(\beta)\\
			&+\sum_{q=1}^{M}\bigg\{\int_{\ps} Q({\alpha})\overline{e_q(\alpha)}{f}(\alpha)ds_\alpha\int_{-\ps}|D_q|
			Q^2(\gamma)\mathrm{v}(\gamma)ds_\gamma\cdot \mm_q\bigg\}\mathrm{v}(\beta) Q({\beta}) e_q(\beta)\\
			&+\sum_{q=1}^{M}\bigg\{\int_{-\ps}|D_q|
			Q^2(\gamma)\mathrm{v}(\gamma)ds_\gamma\cdot \mm_q\int_{\ps}\mathrm{v}(\alpha) Q({\alpha})\overline{e_q(\alpha)}{f}(\alpha)ds_\alpha\bigg\} Q({\beta}) e_q(\beta)\\
			&+\sum_{q=1}^{M}\bigg\{\int_{\ps}\mathrm{v}^\top(\alpha) Q({\alpha})\overline{e_q(\alpha)}{f}(\alpha)ds_\alpha (\mm_q)^\top\int_{-\ps} Q^2(\gamma)\mathrm{v}(\gamma)\mathrm{v}^\top(\gamma)ds_\gamma \mm_q\bigg\}\mathrm{v}(\beta) Q({\beta}) e_q(\beta)\\
			&+\mathcal{O}((k_-L)^{-1/2}).
		\end{align*}
		Consequently, we observe that
		\begin{align*}
				\mathcal{R}({T}^0)\subset\mathop{\mathrm{\bigoplus}}\limits_{1\leq l\leq M}\mathcal{A}_l,
		\end{align*}
		where
		\begin{align}\label{aaaa}
			\begin{aligned}
				\mathcal{A}_l=&\textrm{span}\bigg\{ Q({\beta})e_l(\beta),~ \boldsymbol{\mathrm{e}}_1\cdot\mathrm{v}(\beta) Q({\beta})e_l(\beta),~ \boldsymbol{\mathrm{e}}_2\cdot\mathrm{v}(\beta) Q({\beta})e_l(\beta)\bigg\},
			\end{aligned}
		\end{align}
		which indicates the range dimension of ${T}^0$ is at most $3M$ using Lemma \ref{indep}. Hence, there exist at most $3M$ nonzero eigenvalues for the leading time reversal operator ${T}^0$. 
        
        Now we look for the approximate eigenfunctions in $\mathcal{A}_l$ with $1\leq l\leq M$, respectively. We assert that there exist three approximate eigenfunctions in each $\mathcal{A}_l,1\leq l\leq M$. Since $\mathcal{A}_l$ is an approximate invariant subspace for ${T}^0$, as shown in appendix \ref{app}, we get the matrix representation,
\[{T}^0\left(g_l,h_{l,1},h_{l,2}\right)=\left(g_l,h_{l,1},h_{l,2}\right)\mathbb{A}_l+\mathcal{O}((k_-L)^{-1/2}),\]
		where $\mathbb{A}_l$ is a $3\times3$ matrix defined in \eqref{alll}. For simplicity, we denote the three eigenvalues of $\mathbb{A}_l$ by $\zeta_{l,q}$ and the corresponding eigenfunctions by $v_{l,q}, ~ q=1,2,3$, i.e.,
		\[\mathbb{A}_lv_{l,q}=\zeta_{l,q}v_{l,q},\quad q=1,2,3.\]
		Thus, we get the three approximate eigenfunctions in $\mathcal{A}_l$ as follows
		\begin{align}\label{glp}
			f_{l,q}=v_{l,q}^1g_l+v_{l,q}^2h_{l,1}+v_{l,q}^3h_{l,2},\quad q=1,2,3,
		\end{align}
		where $v_{l,q}^j$ denotes the $j$-th element of $v_{l,q}$.
	\end{proof}
	
	In summary, we have proved that in the regime of small and well-resolved sound-hard buried particles, the
	leading time reversal operator $T^0$ admits $3M$ eigenvalues, and	the approximate eigenfunctions of ${T}^0$ are the far fields corresponding to linear combinations of monopole and dipole sources located at the centers of the unknown particles.
	
	\begin{theorem}\label{gloo}
		For $1\leq l\leq M$, the incident Herglotz wave associated with the approximate eigenfunction ${f}_{l,q},q=1,2,3$ given by \eqref{glp} selectively focuses on the $l$th particle.
	\end{theorem}
    
	The proof follows from Lemma \ref{decay} since the transmitted wave corresponding to incident waves $u_{f_{l,q}}^i(x)=\int_{\ps}{f}_{l,q}(\beta)\mathrm{e}^{\mi k_+\beta\cdot x}ds_\beta$ is
	\begin{align*}
		u_{f_{l,q}}^{trans}(x)&=\int_{\ps}{f}_{l,q}(\beta)Q(\beta)\mathrm{e}^{\mi k_-\mv(\beta)\cdot x}ds_\beta\\&=v_{l,q}^1\int_{\ps}Q^2(\beta)\mathrm{e}^{\mi k_-\mv(\beta)\cdot (x-s_l)}ds_\beta+\sum_{j=1,2}v_{l,q}^{j+1}\boldsymbol{\mathrm{e}}_j\cdot\int_{\ps}\mv(\beta)Q^2(\beta)\mathrm{e}^{\mi k_-\mv(\beta)\cdot (x-s_l)}ds_\beta,
	\end{align*}
	which decays rapidly away from the $l$-th particle as $\mathcal{O}((k_-|x-s_l|)^{-\frac{1}{2}})$.
	\begin{remark}(Three dimensional case)
		Clearly, similar selective focusing results hold for the three dimensional problem. More precisely, one can show that each small sound-soft and sound-hard buried particle in a layered medium gives rise to 1 and 4 eigenvalues, respectively.
	\end{remark}
	\begin{remark} (Impedance and Penetrable cases).
		Although we only consider impenetrable scatterers of Dirichlet and Neumann types here, we note that because of the similar structure of the asymptotic expansions \eqref{farps} for impenetrable scatterers of impedance type \cite{liRecoveringMultiscaleBuried2015} and \eqref{asyfar5} for penetrable scatterers \cite{ammariMUSICAlgorithmLocating2005}, the results obtained for obstacles carrying a Dirichlet or Neumann boundary condition can also be applied to scatterers of impedance type and penetrable scatterers.
	\end{remark}
    
    As we can see from \eqref{b0}, when the wavenumber is relatively small, the decaying rate of the oscillation integral is slow, resulting in an unsatisfactory imaging result. Due to this limitation, in the following, using our time reversal method as an initial indicator, we consider recovering multiple extended buried obstacles for higher resolution based on Bayesian inversion technique in the low-frequency region.
	\section{Bayesian inversion for the shape of extended buried obstacles}
In this section,  we assume that the boundary $\partial D_i$ of the obstacle $D_i$ can be represented as
		\begin{align*}
			\partial D_i:=(c_{i1},c_{i2})+r_i(\theta)(\cos\theta,\sin\theta),
			\quad \theta\in(0,2\pi),1\leq i\leq M,
		\end{align*}
		where $(c_{i1},c_{i2})$ denotes the position of $D_i$, and $r_i(\theta)$ is given in the following Fourier expansion
		\begin{align}\label{fouri}
			r_i(\theta)=a_{i0}+\sum_{j=1}^N\left(a_{ij}\cos j\theta+b_{ij}\sin j\theta\right),\quad j=1,2,\cdots,N.
		\end{align} 
        Thus, the unknown parameters can be denoted as a finite set $q$:
	\begin{align*}
		q = (q_1,q_2,\cdots,q_M)\in \mathbb{R}^{M(3+2N)},
	\end{align*}
	where $q_i=(c_{i1},c_{i2},a_{i0},a_{ij},b_{ij}),i=1,\cdots,M,j = 1,\cdots,N$.
	% The space $\mathcal{X}$ is equipped with the norm $\|A\|_\mathcal{X}=\sqrt{(A,A)}$ for $A\in \mathcal{X}$ and
	% $(A,A)$ means the standard inner-product in $\mathcal{X}$. 
    Suppose that $q$ obeys the prior Gaussian distribution with mean $\mathbf{m}_{pr}$ and variance $\Sigma_{pr}$. Meanwhile, each variable in $q$ is independent. Let $\pi_{pr}$ and $\mu_{pr}$ denote the prior density and measure of $q$, respectively. Since the obstacles can be determined by the finite set $q$,
    % Due to the relationship between the obstacle $D$ and the far field measurement $u^\infty(\hat{x},d)$ with fixed $d$, 
    the direct scattering operator can be defined as
	% \begin{align}\label{reladf}
	% 	u^\infty(\hat{x},d) = \mathcal{G}(D).
	% \end{align}
	% , \eqref{reladf} can be rewritten as
	\begin{align*}
		u^\infty(\hat{x},d) = \mathcal{G}(q),
	\end{align*}
	where the nonlinear observation operator $\mathcal{G}$ is an abstract map from the space of parameters to the space of far field measurements. Taking the noise generated in the measurement procedure into account, the statistical model in this section can be presented by
	\begin{align*}
		y = \mathcal{G}(q)+\eta.
	\end{align*}
	The observation noise $\eta$ is assumed to satisfy $\eta\sim N(0,\sigma^2 \mathbf{I}_{K\times K})$ and $y\in \mathbb{C}^K$. $\mathbf{I}_{K\times K}$ is the identity matrix where $K$ is the number of observation directions. 
    
    Let $\pi_{post}$ and
	$\mu_{post}$ denote the posterior density and measure of $q$, respectively. 
	The purpose of the Bayesian method is to solve the posterior measure
	$\mu_{post}$ of $q$. According to the Bayesian formula \cite{bui-thanhAnalysisInfiniteDimensional2014}, the posterior density is given as follows
	\begin{align*}
		\pi_{post}(q)=\frac{\rho(y-\mg(q))\pi_{pr}(q)}{\int_{\mathbb{R}^P} \rho(y-\mg(q'))\pi_{pr}(q')dq'}
	\end{align*}
	where $\rho(y-\mg(q))=\rho(y|q)$, $y|q\sim N(\mg(q),\sigma^2\mathbf{I}_{K\times K})$ is the likelihood function, which gives the conditional density of $y$ by given $q$. Then 
	\begin{align*}
		\begin{aligned}
			\pi_{post}(q)&\propto \rho(y-\mg(q))\pi_{pr}(q)\\
			& \propto \mathrm{exp}\left(-\frac{1}{2\sigma^2}\|y-\mg(q)\|^2\right)\pi_{pr}(q).
		\end{aligned}
	\end{align*}
	where the norm stems from the standard inner product in $\mathbb{C}^K$. 
    
    Finally, we adopt the MCMC algorithm to explore the posterior distribution of the unknown set $q$, which is given in Algorithm \ref{mcmc}.
	\begin{algorithm}[htbp]
		\caption{MCMC algorithm of recover multiple buried obstacles.}\label{mcmc}
		\begin{algorithmic}
			\STATE\begin{itemize}
				\item[Step 1\quad] Set a scalar $\beta\in(0,1)$ and the number of iterations $T_{max}$.
				\item[Step 2\quad] Initialize $q^{(0)}=(q^{(0)}(1),q^{(0)}(2),\cdots,q^{(0)}(M))$, the number of obstacles $M$ and the initial guess is obtained by the time reversal method.
				\item[Step 3\quad] Update $\widetilde{q}(i)$ from $q^{(k)}(i),i=1,2,\cdots,M,k=1,2,\cdots,T_{max}$
				\begin{align*}
					\widetilde{q}(i) &= \mathbf{m}_{pr}(i)+\sqrt{1-\beta^2}(q^{(k)}(i)-\mathbf{m}_{pr}(i))+\beta\omega,~~\omega\sim N(0,1),\\
					\widetilde{q} &= (\widetilde{q}^{(0)}(1),\widetilde{q}^{(0)}(2),\cdots,\widetilde{q}^{(0)}(M)),\\
					q^{(k)} &= (q^{(k)}(1),q^{(k)}(2),\cdots,q^{(k)}(M)),
				\end{align*}
				where $\mathbf{m}_{pr}(i)$ is the mean value of $q(i)$.
				\item[Step 4\quad] Calculate acceptance ratio $\alpha$
				\begin{align*}
					\alpha(\widetilde{q},q^{(k)})=\operatorname{min}\left\{1,\operatorname{exp}\left(\frac{1}{2\sigma^2}\|y-\mg(q^{(k)})\|_\my^2-\frac{1}{2\sigma^2}\|y-\mg(\widetilde{q})\|_\my^2\right)\right\}.
				\end{align*}
				\item[Step 5\quad] Draw $\widetilde{\alpha}\sim U(0,1)$. If $\widetilde{\alpha}<\alpha(\widetilde{q},q^{(k)})$, set $q^{(k+1)}=\widetilde{q}$; else, reject $\widetilde{q}$ and keep $q^{(k+1)}=q^{(k)}$.
				\item[Step 6\quad]  Repeat Step 3-5 until $T_{max}$ times and compute the mean of last 100 samples.
			\end{itemize}
		\end{algorithmic}
	\end{algorithm}
	\begin{remark}
		To speed up the convergence of the sampling process, we can utilize the Hessian information in the inversion part. Readers are referred to \cite{flathFastAlgorithmsBayesian2011} for some Hessian-based sampling methods. 
	\end{remark}
    \begin{remark}
            With the aid of the Fr\'echet differentiability of $\mathcal{G}$ \cite{liImagingBuriedObstacles2021}, the well-posedness of the Bayesian framework in layered media follows a similar spirit in the homogeneous case \cite{bui-thanhAnalysisInfiniteDimensional2014}.
        \end{remark}
	
	\section{Numerical experiments}\label{exper}
	In this section, we test our global and selective focusing results as well as the efficiency of the time reversal method cooperated with the Bayesian technique in several numerical examples. Here, the far field data is obtained by the forward solver \cite{laiFastSolverMultiparticle2014} based on integral equations. For the inversion part, to avoid inverse crime and show the robustness of the algorithm, we add 5\% Gaussian noise to the simulated far field data. Once the numerical time reversal operator $T$ is found, the eigensystems are obtained through
	the ‘$eig$’ command in MATLAB, and Herglotz wave $u^i_f$ and its transmission wave $u_f^{trans}$ are evaluated in a straightforward manner.
	Throughout all the examples, 
	the emission and reception directions are obtained by discretizing  $\ps=[\pi+\theta_c,2\pi-\theta_c]$ and $-\ps=[\theta_c,\pi-\theta_c]$ uniformly using 180 points, respectively. 
	
	\subsection{Selective imaging results of multiple particles}\label{selec}
    In this subsection, we present several numerical examples to verify our selective focusing results in Section \ref{secle} and illustrate their applicability to imaging multiple particles under different boundary conditions.  We choose wavenumbers to be $k_+=40$ and $k_-=30$ here.
    
    First, to investigate the number of significant eigenvalues in the asymptotic region, we consider a disk located at $(-5,-10)$ with radius $10^{-4}$. Figure \ref{ex2}(a) and (b) show the first thirty eigenvalues of the time reversal operator in sound-soft and sound-hard cases, respectively. It can be easily seen that a sound-soft particle gives only one significant eigenvalue while a sound-hard particle gives three. Note that the difference of their magnitude between sound-soft and sound-hard cases can be derived from the factor $(\ln \rho)^{-1}$ in \eqref{farps} and $\rho^2$ in \eqref{asyfar5}.  
	
	Then we consider two small obstacles which are a kite-shaped obstacle with $a = 0.0325,b=0.05,c=0.075$ and a starfish-shaped obstacle with $a=0.1,b=0.03$ given in Table \ref{tab}. Under sound-soft boundary condition, Figure \ref{ex2}(c) and (d) show the absolute values of the Herglotz wave with the eigenfunction corresponding to the first and second eigenvalue as its kernel, respectively. Here the sampling region is chosen to be $[-10, 10]\times[-20, 0]$.
	We observe that each particle gives one significant eigenvalue, and the Herglotz wave function associated with the corresponding eigenvalue achieves selective focusing at the location of the unknown particle. As shown in Figure \ref{ex4}(a-f), similar selective focusing results also hold for sound-hard boundary condition, which is consistent with the analysis in Section \ref{secle}.
	
	Furthermore, in Figure \ref{ex10}(a), we consider 9 disks with radius $r=0.1$ uniformly distributed on a circle with radius $R=7$ located at $(0,-10)$. Figure \ref{ex10}(b) and (c) demonstrate the superposition of selective focusing results of sound-soft particles using the first 9 eigenfunctions and sound-hard particles using the 
 first 27 eigenfunctions, respectively.
	
	\begin{figure}[htp!]
		\centering
		\subfigure[]{\includegraphics[scale=0.3]{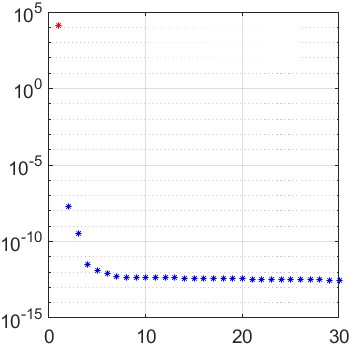}}
		\subfigure[]{\includegraphics[scale=0.3]{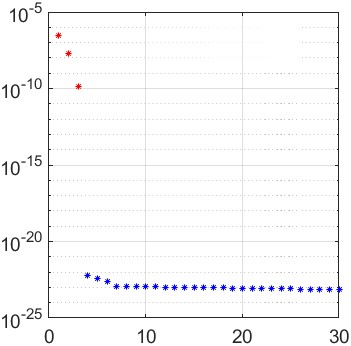}}
		\subfigure[]{\includegraphics[scale=0.3]{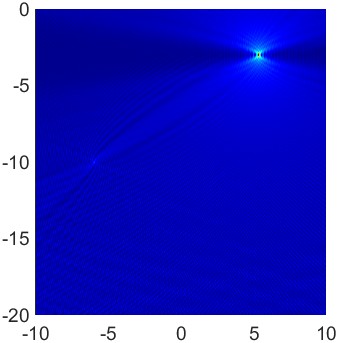}}
		\subfigure[]{\includegraphics[scale=0.3]{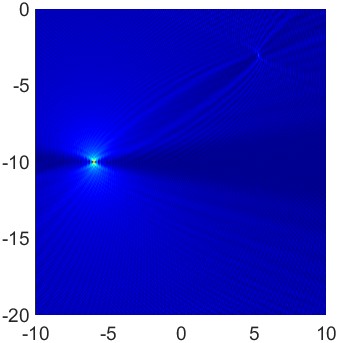}}
		\caption{Eigenvalues of the time reversal operator. (a) Sound-soft boundary condition. (b) Sound-hard boundary condition. Imaging of two sound-soft particles. (c)(d) The Herglotz wave with the eigenfunction of the first and second eigenvalues, respectively.}\label{ex2}
	\end{figure}

	\begin{figure}[htp!]
		\centering
		\subfigure[]{\includegraphics[scale=0.20]{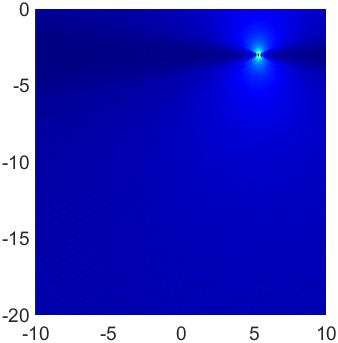}}
		\subfigure[]{\includegraphics[scale=0.20]{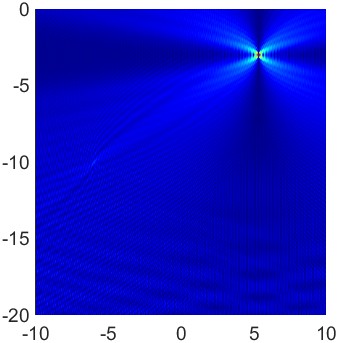}}
		\subfigure[]{\includegraphics[scale=0.20]{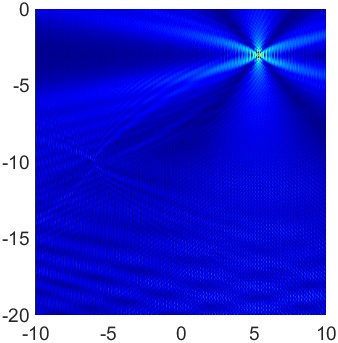}}
		\subfigure[]{\includegraphics[scale=0.20]{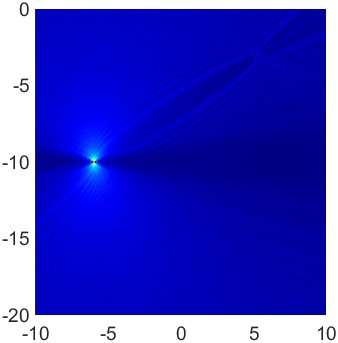}}
		\subfigure[]{\includegraphics[scale=0.20]{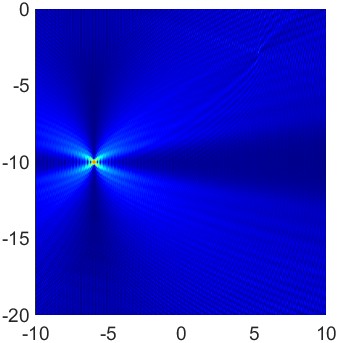}}
		\subfigure[]{\includegraphics[scale=0.20]{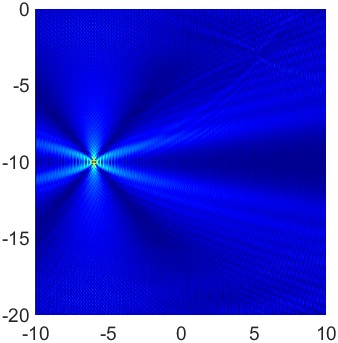}}
		\caption{Imaging of two sound-hard particles. The Herglotz wave with the eigenfunction of the first six eigenvalues.}\label{ex4}
	\end{figure}
	
	\begin{figure}[htp!]
		\centering
		\subfigure[]{\includegraphics[scale=0.35]{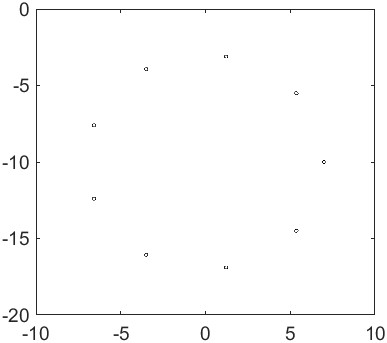}}
		\subfigure[]{\includegraphics[scale=0.35]{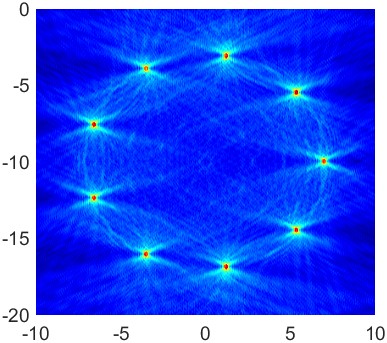}}
		\subfigure[]{\includegraphics[scale=0.35]{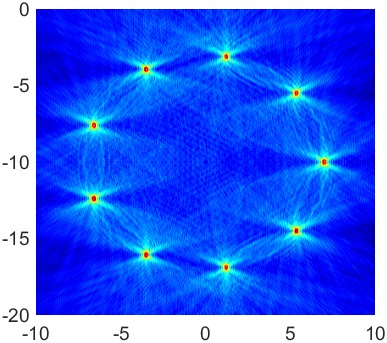}}
		\caption{Imaging of 9 particles. (a) Exact locations. (b) Imaging of sound-soft particles. (c) Imaging of sound-hard particles.}\label{ex10}
	\end{figure}
	\subsection{Imaging of extended obstacles based on time reversal operator}\label{gle}
	We consider two extended obstacles, which are a kite-shaped obstacle with $a = 1.3,b=2,c=3$ and a starfish-shaped obstacle with $a=2,b=0.6$ given in Table \ref{tab}. The geometry is shown in Figure \ref{ex1}(a). In this extended case, there are no explicitly significant eigenvalues anymore. Figures \ref{ex1}(b) and (c) demonstrate the imaging results of the Herglotz wave under sound-soft and sound-hard cases, respectively, whose kernels are both combinations of eigenfunctions corresponding to the first eighty eigenvalues. It can be seen that the location and approximate shape of each obstacle are well recovered.
	\begin{table}[htp!]
		\caption{Parametrization of the obstacles}
		\begin{center}
			\begin{tabular}{ll}
				\hline
				Type & Parametrization\\
				\hline
				\text{Kite:}\quad & $x(t) = (6,-3)+(a\cos 2t+b\cos t,c\sin t), \qquad t\in[0,2\pi]$\\
				\text{Starfish:}\quad& $x(t)=(-6,-10)+(a+b\cos(5(t-\pi/2)))(\cos t,\sin t),
				\qquad t\in [0,2\pi]$\\
				% \hline
                \text{Pear:}\quad& $x(t)=(6,-3)+(2+0.5\cos 3t)(\cos t,\sin t),
				\qquad t\in [0,2\pi]$\\
				\text{Apple:}\quad& $x(t)=(-6,-10)+\frac{2+1.8\cos t+0.2\sin 2t}{1+0.75\cos t}(\cos t,\sin t),
				\qquad t\in [0,2\pi]$\\
				\hline
			\end{tabular}
		\end{center}
		\label{tab}
	\end{table} 
	\begin{figure}[htp!]
		\centering
		\subfigure[]{\includegraphics[scale=0.35]{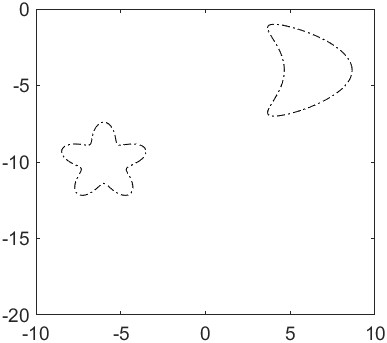}}
		\subfigure[]{\includegraphics[scale=0.35]{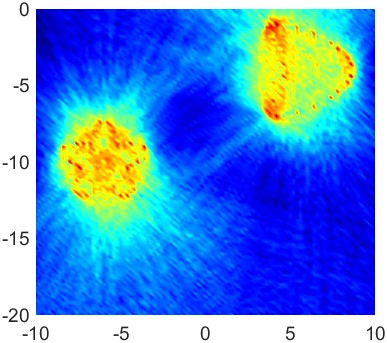}}
		\subfigure[]{\includegraphics[scale=0.35]{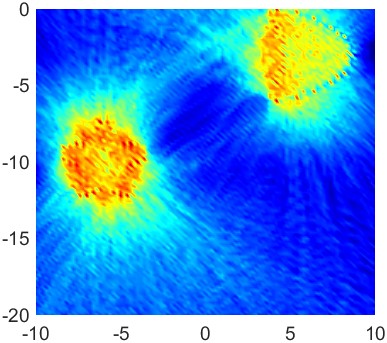}}
		\caption{Imaging of two extended obstacles. (a) The black dotted line represents the exact shape. (b) The Herglotz wave with eigenfunctions corresponding to the first {eighty} eigenvalues of sound-soft obstacles. (c) The Herglotz wave with eigenfunctions corresponding to the first eighty eigenvalues of sound-hard obstacles.}\label{ex1}
	\end{figure}
	\subsection{Reconstruction of extended obstacles improved by Bayesian technique} 
    % As we mentioned above, when the wavenumber is relatively small, the decaying speed of oscillation integral is a bit slow, resulting in the unsatisfactory global imaging result. Due to this limitation,
    In this subsection, using our time reversal method as an initial indicator, we consider recovering multiple extended buried particles more accurately based on the Bayesian inversion technique in the low-frequency region with $k_+=4$ and $k_-=3$. First, we consider a pear-shaped and an apple-shaped obstacle whose parameterizations are given in Table \ref{tab}. Figure \ref{ex5}(a) shows the Herglotz wave with a combination of the first several eigenfunctions. Figure \ref{ex5}(b) shows the exact scatterers and the initial guess from the time reversal method. One can see that the result has a very low resolution; therefore, we apply the Bayesian inversion scheme to improve the imaging. We choose the number of Fourier expansion terms to be $N = 3$  in \eqref{fouri} and the scale parameter $\beta = 0.01$ in Algorithm \ref{mcmc}. Adopting 2000 iterations in the MCMC algorithm, Figure \ref{ex5}(c) plots the Markov chains of $a_{i0},i=1,2$ in \eqref{fouri} during the sampling process. The Bayesian inversion result is demonstrated in Figure \ref{ex5}(d), from which we can see the result is greatly improved compared with the time reversal result.
    Figure \ref{ex5}(e-h) show similar results hold for sound-hard obstacles. 

	% \begin{table}[htp!]
	% 	\caption{Parametrization of the obstacles}
	% 	%		\renewcommand\arraystretch{1.5}
	% 	\begin{center}
	% 		\begin{tabular}{ll}
	% 			%		\vspace{0.5em}
	% 			\hline
	% 			Type & Parametrization\\
	% 			\hline
	% 			\text{Pear:}\quad& $x(t)=(6,-3)+(2+0.5\cos 3t)(\cos t,\sin t),
	% 			\qquad t\in [0,2\pi]$\\
	% 			\text{Apple:}\quad& $x(t)=(-6,-10)+\frac{2+1.8\cos t+0.2\sin 2t}{1+0.75\cos t}(\cos t,\sin t),
	% 			\qquad t\in [0,2\pi]$\\
	% 			\hline
	% 		\end{tabular}
	% 	\end{center}
	% 	\label{tab2}
	% \end{table} 
	\begin{figure}[htp!]
		\centering
		\subfigure[]{\includegraphics[scale=0.25]{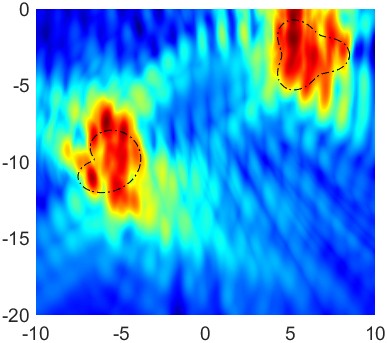}}
		\subfigure[]{\includegraphics[scale=0.25]{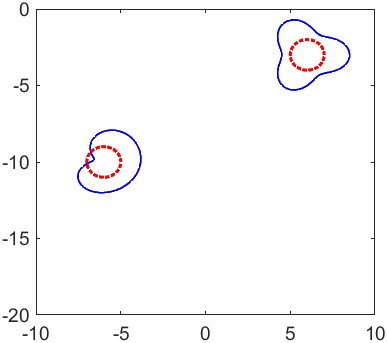}}
		\subfigure[]{\includegraphics[scale=0.25]{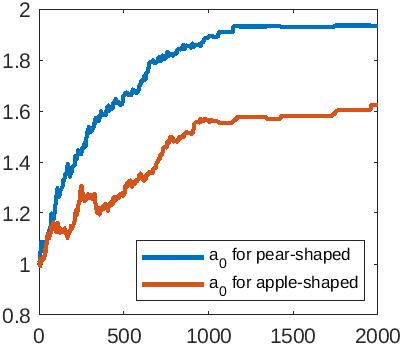}}
		\subfigure[]{\includegraphics[scale=0.25]{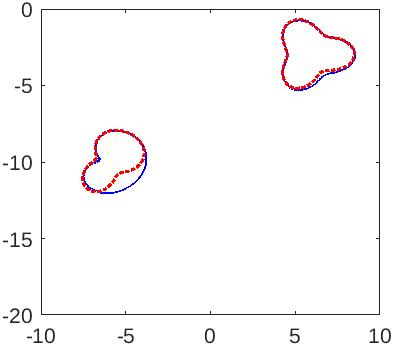}}
		\subfigure[]{\includegraphics[scale=0.25]{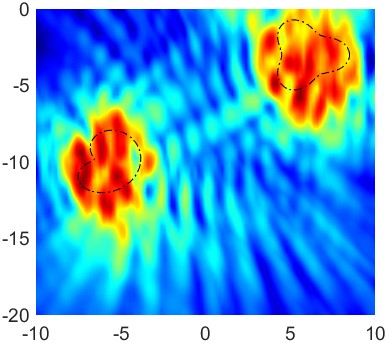}}
		\subfigure[]{\includegraphics[scale=0.25]{figure/extn2.jpg}}
		\subfigure[]{\includegraphics[scale=0.25]{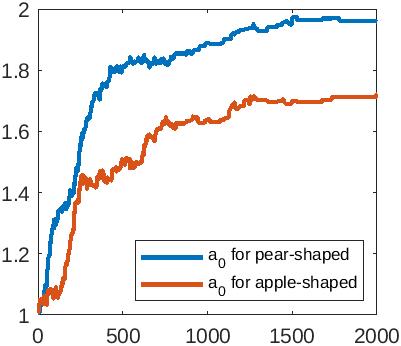}}
		\subfigure[]{\includegraphics[scale=0.25]{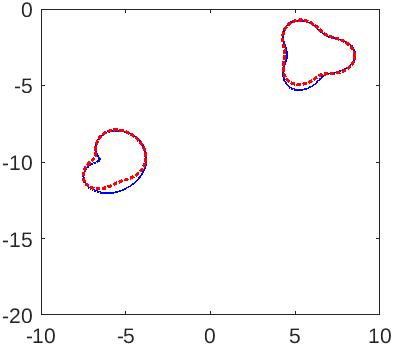}}
		\caption{Imaging and Bayesian inversion of two extended obstacles. (a) Global imaging result. (b) Exact shape and initial guess. (c) Markov chains for $a_{i0}$. (d) The result of Bayesian inversion. 
 (e-h) Similar results for sound-hard obstacles}\label{ex5}
	\end{figure}

	\section{Conclusion}
	In this work, we have developed the global and selective focusing method of multiple particles in a layered medium based on the far field time reversal model. By a combination of the integral equation method and the decay property of the oscillation integral, we show the effective imaging based on the eigensystem of the time reversal operator in a general case. In the region of small and well-resolved particles, we show that each sound-soft and sound-hard particle gives rise to one and three significant eigenvalues, respectively, whose corresponding eigenfunctions can be used to achieve selective focusing. Furthermore, a layered Bayesian scheme is proposed for better shape reconstruction with the time reversal result as an initial guess. Numerical experiments are demonstrated to show the applicability of these inversion methods. Our future work includes extending the time reversal method to general inhomogeneous media and multiple buried elastic obstacles.
	
	\appendix
    
    \section{Independence of $g_l$ and $h_{l,q}$}
    \begin{lemma}\label{indep}
		Under the notation of \eqref{gl}, the functions $g_l$ and $h_{l,q}$ for $l=1,\cdots,M$ and $q=1,2$ are linearly independent.
	\end{lemma}
	\begin{proof}
		Suppose that there exist complex coefficients $(a_l)_{l=1,\cdots,M}$ and
		$(b_{l,q})_{l=1,\cdots,M}^{q=1,2}$ such that
		\begin{align*}
			\sum_{l=1}^M a_l g_l(\beta)+\sum_{l=1}^M\sum_{q=1}^{2}\mi 
			k_-b_{l,q}h_{l,q}(\beta)=0,
			\quad \forall\beta\in\ps,
		\end{align*}
		which amounts to saying that the function 
		$$u(x):= \sum_{l=1}^M a_l G(x,s_l)+\sum_{l=1}^M\sum_{q=1}^{2}b_{l,q}\nabla G(x,s_l)\cdot \boldsymbol{\mathrm{e}}_q,\quad x\in\lowsp,$$ corresponding to the incident field  of a combination of monopole and dipole sources has a vanishing far field in the upper half-space $\upsp$. By Rellich’s lemma in a layered medium similar to Lemma 2.12 in  \cite{coltonInverseAcousticElectromagnetic2019}, $u$ must vanish everywhere in $\lowsp$, so that
		\begin{align}\label{sin}
			\sum_{l=1}^M a_l G(x,s_l)+\sum_{l=1}^M\sum_{q=1}^{2}b_{l,q}\nabla G(x,s_l)\cdot \boldsymbol{\mathrm{e}}_q=0,\qquad \forall x\in\lowsp\notin\{s_1,\cdots,s_M\}.
		\end{align}
		Fix $l\in\{1,\cdots M\}$ and choose $x=s_l+\rho\hat{x}$ with $\hat{x}\in\ms$ and
		$0\leq\rho\leq\rho*$ small enough. Multiplying \eqref{sin} by $\rho$ and taking the limit as $\rho\rightarrow0$, we note that the only non-vanishing contribution comes from the most singular term. More precisely, the dipole term corresponding to $l=m$ gives \[\lim\limits_{\rho\rightarrow0}\rho\nabla G(s_m+\rho\hat{x},s_m)\cdot \boldsymbol{\mathrm{e}}_q=-\frac{1}{2\pi}(\hat{x}\cdot \boldsymbol{\mathrm{e}}_q).\] Therefore, $\hat{x}\cdot(\sum_{q=1}^{2}b_{m,q}\boldsymbol{\mathrm{e}}_q)=0$. Since $\hat{x}\in\ms$
		is arbitrary and since the vector columns $\boldsymbol{\mathrm{e}}_1,\boldsymbol{\mathrm{e}}_2$ are linearly independent, the above relation yields $b_{m,1}=b_{m,2}=0$ for all $m\in\{1,\cdots,M\}$. Similarly, we can easily show that $a_1=\cdots=a_M=0$.
	\end{proof}
	\section{Matrix representation of ${T}^0$ in $\mathcal{A}_l$}\label{app}
	By the insertion of \eqref{gl} into the definition of ${T}^0$, we obtain that,
	
	\begin{align*}
		{T}^0g_l=&\bigg\{(|D_l|)^2\int_{-\ps}
		Q^2(\gamma)ds_\gamma\int_{\ps}Q^2({\alpha})ds_\alpha\bigg\} Q({\beta})e_l(\beta)\\
		&+\bigg\{\int_{\ps}Q^2({\alpha})ds_\alpha\int_{-\ps}|D_l|
		Q^2(\gamma)\mathrm{v}(\gamma)ds_\gamma\cdot \mm_l\bigg\}\mathrm{v}(\beta)Q({\beta}) e_l(\beta)\\
		&+\bigg\{\int_{-\ps}|D_l|
		Q^2(\gamma)\mathrm{v}(\gamma)ds_\gamma\cdot \mm_l\int_{\ps}\mathrm{v}(\alpha)Q^2({\alpha})ds_\alpha\bigg\}Q({\beta}) e_l(\beta)\\
		&+\bigg\{\int_{\ps}\mathrm{v}(\alpha)Q^2({\alpha})ds_\alpha(\mm_l)^\top\int_{-\ps}Q^2(\gamma)\mathrm{v}(\gamma)\mathrm{v}^\top(\gamma)ds_\gamma \mm_l\bigg\}\mathrm{v}(\beta)Q({\beta}) e_l(\beta)\\
		&+\mathcal{O}((k_-L)^{-1/2})\\
		=&a^l_{1,1}g_l+a^l_{2,1}h_{l,1}+a^l_{3,1}h_{l,2}+\mathcal{O}((k_-L)^{-1/2}),
	\end{align*}
    and
	\begin{align*}
		{T}^0h_{l,j}=&\bigg\{(|D_l|)^2\int_{-\ps}
		Q^2(\gamma)ds_\gamma\int_{\ps}\boldsymbol{\mathrm{e}}_j\cdot\mathrm{v}(\alpha)Q^2({\alpha})ds_\alpha\bigg\} Q({\beta})e_l(\beta)\\
		&+\bigg\{\int_{\ps}\boldsymbol{\mathrm{e}}_j\cdot\mathrm{v}(\alpha)Q^2({\alpha})ds_\alpha\int_{-\ps}|D_l|
		Q^2(\gamma)\mathrm{v}(\gamma)ds_\gamma\cdot \mm_l\bigg\}\mathrm{v}(\beta)Q({\beta}) e_l(\beta)\\
		&+\bigg\{\int_{-\ps}|D_l|
		Q^2(\gamma)\mathrm{v}(\gamma)ds_\gamma\cdot \mm_l\int_{\ps}Q^2({\alpha})\mathrm{v}(\alpha)\mathrm{v}^\top(\alpha)ds_\alpha\boldsymbol{\mathrm{e}}_j \bigg\}Q({\beta}) e_l(\beta)\\
		&+\bigg\{\boldsymbol{\mathrm{e}}_j^\top\int_{\ps}Q^2({\alpha})\mathrm{v}(\alpha)\mathrm{v}^\top(\alpha)ds_\alpha (\mm_l)^\top\int_{-\ps}Q^2(\gamma)\mathrm{v}(\gamma)\mathrm{v}^\top(\gamma)ds_\gamma \mm_l\bigg\}\mathrm{v}(\beta)Q({\beta}) e_l(\beta)\\
		&+\mathcal{O}((k_-L)^{-1/2})\\
		=&a^l_{1,1+j}g_l+a^l_{2,1+j}h_{l,1}+a^l_{3,1+j}h_{l,2}+\mathcal{O}((k_-L)^{-1/2}),\quad j=1,2.
	\end{align*}
	
	Thus,\[{T}^0\begin{pmatrix}
		g_l&h_{l,1}&h_{l,2}
	\end{pmatrix}=\begin{pmatrix}
		g_l&h_{l,1}&h_{l,2}
	\end{pmatrix}\mathbb{A}_l+\mathcal{O}((k_-L)^{-1/2}),\]
    
	where \begin{align}\label{alll}
	    \mathbb{A}_l=\begin{pmatrix}
		a^l_{1,1}&a^l_{1,2}&a^l_{1,3}\\
		a^l_{2,1}&a^l_{2,2}&a^l_{2,3}\\
		a^l_{3,1}&a^l_{3,2}&a^l_{3,3}
	\end{pmatrix}.
	\end{align}
	\bibliographystyle{plain}
	\bibliography{bib}
\end{document}